\documentclass[11pt,letterpaper]{article}
\usepackage[margin=1in]{geometry}
\usepackage[T1]{fontenc}
\usepackage{lmodern}
\usepackage{microtype}
\usepackage{amsmath,amssymb,amsthm,mathtools,bm}
\usepackage{authblk}
\usepackage{algorithm,algorithmic}
\usepackage{tikz,array,siunitx}
\usepackage[font=small]{subfig}
\usepackage{textcomp,url,verbatim,graphicx,tablefootnote,xcolor}
\usepackage{booktabs,multirow}
\usepackage[shortlabels]{enumitem}
\usepackage[normalem]{ulem}
\usepackage[authoryear,round]{natbib}
\usepackage[hidelinks]{hyperref}
\usepackage{float}

\theoremstyle{plain}
\newtheorem{theorem}{Theorem}
\newtheorem{lemma}{Lemma}

\newtheorem{assumption}{Assumption}
\theoremstyle{definition}
\newtheorem{definition}{Definition}
\theoremstyle{remark}
\newtheorem{remark}{Remark}
\DeclareMathOperator*{\argmin}{arg\,min}

\newcommand{\proj}{\text{Proj}}
\newcommand{\rgrad}{\mathrm{grad}\,}
\newcommand{\St}{\mathrm{St}}

\useunder{\uline}{\ul}{}

\DeclareMathOperator{\Proj}{Proj}

\newcommand{\tcb}[1]{{#1}}
 \bibpunct[, ]{(}{)}{,}{a}{}{,}%
 \def\R{\mathbb{R}}
 \def\M{\mathcal{M}}
 \def\Retr{\mathrm{Retr}}
 \def\dom{\mathrm{dom}}

\title{An Inexact Manifold Proximal Linear Algorithm with Adaptive Stopping Criteria}
\author[1]{Zhong Zheng}
\author[2]{Xin Yu}
\author[3]{Shiqian Ma\textsuperscript{*}}
\author[2]{Lingzhou Xue\textsuperscript{*}}
\affil[1]{Department of Applied and Computational Mathematics and Statistics, University of Notre Dame}
\affil[2]{Department of Statistics, The Pennsylvania State University}
\affil[3]{Data Science and AI Institute and Department of Applied Mathematics and Statistics, Johns Hopkins University}
\affil[ ]{\texttt{zzheng2@nd.edu}, \texttt{xmy5152@psu.edu}, \texttt{sqma@jhu.edu}, \texttt{lzxue@psu.edu}}
\affil[ ]{\textsuperscript{*}Corresponding authors.}

\date{}
\hypersetup{
  pdftitle={An Inexact Manifold Proximal Linear Algorithm with Adaptive Stopping Criteria},
  pdfauthor={Zhong Zheng, Xin Yu, Shiqian Ma, Lingzhou Xue},
  pdfkeywords={Nonconvex Optimization, Composite Optimization, Manifold Optimization, Proximal Linear Algorithm, Inexactness, Sparse Principal Component Analysis, Sparse Spectral Clustering}
}

\begin{document}
\maketitle
\begin{abstract}
This paper proposes a new inexact manifold proximal linear (IManPL) algorithm for solving nonsmooth, nonconvex composite optimization problems over an embedded submanifold. At each iteration, IManPL solves a convex subproblem inexactly, guided by two adaptive stopping criteria. Under the stated assumptions, we establish convergence guarantees and an $O(\epsilon^{-3})$ first-order oracle complexity bound for APG subproblem solves, matching the best-known order for this class of problems. Numerical experiments on sparse spectral clustering and sparse principal component analysis demonstrate that our methods outperform existing approaches.
\end{abstract}
\noindent\textbf{Keywords:} Nonconvex Optimization, Composite Optimization, Manifold Optimization, Proximal Linear Algorithm, Inexactness, Sparse Principal Component Analysis, Sparse Spectral Clustering

\medskip
\noindent\textbf{Funding:} This research was supported by NIH grant 1R01GM152812, NSF grants CCF-2007823, ECCS-2326591, and CCF-2626366, and ONR grant N00014-24-1-2705.

\section{Introduction}
In this paper, we focus on the following nonsmooth and nonconvex composite optimization
problem over a compact embedded submanifold: 
\begin{equation}\label{problem_single}
	\min_{z \in \mathcal{M}} F(z) :=  f(z)+h(c(z)).
\end{equation}
Here $\mathcal{M}$ denotes a compact submanifold embedded in $\R^n$, $f:\R^n\rightarrow \mathbb{R}, c:\R^n\rightarrow \mathbb{R}^m$ are smooth mappings, and $h:\mathbb{R}^m\rightarrow \mathbb{R}$ is a convex function with a tractable proximal mapping. Here, the convexity and smoothness of the functions are interpreted as the functions being considered in the ambient space. Many important problems in statistics, machine learning, and signal processing can be cast into the form of problem \eqref{problem_single}. Examples include sparse principal component analysis (SPCA) \citep{jolliffe2003modified,zou2018selective,yang2025theoretical}, fair PCA \citep{samadi2018price,zalcberg2021fair,xu2023efficient2}, sparse canonical correlation analysis (SCCA) \citep{deng2024oracle,chen2019alternating,hardoon2011sparse}, sparse spectral clustering (SSC) \citep{lu2016convex,lu2018nonconvex,wang2022manifold}, orthogonal dictionary learning \citep{sun2016complete,sun2016complete2,qu2014finding,demanet2014scaling,spielman2012exact},
and beamforming design \citep{liu2024survey}. Below, we highlight two representative applications.
\begin{itemize}
\item \textbf{Sparse spectral clustering (SSC)}. Spectral clustering \citep{chung1997spectral} is a graph-based clustering method in unsupervised learning. Given $N$ data points that belong to $r\leq N$ classes and a nonnegative symmetric similarity matrix $\hat{S}\in\mathbb{R}^{N\times N}$, SSC forms a normalized Laplacian matrix $S :=  I_N - D^{-1/2} \hat{S} D^{-1/2}$ where $D:=  \mbox{diag}(d_1,d_2\ldots d_N)$, in which $d_i$ represents the sum of the elements in the $i-$th row of $\hat{S}, i = 1,2\ldots N$. We assume $d_i>0$ for all $i=1,\ldots,N$ so that $D^{-1/2}$ is well defined. In \cite{wang2022manifold}, the main step of SSC is to solve the following optimization problem: 
	\begin{equation}\label{eq_ssc}
		\min_{U\in\St(N,r)} \left\langle U, SU\right\rangle + \kappa\|\mbox{vec}(UU^\top )\|_1.
	\end{equation}
    Here, $\St(N,r)=\{Z\in\R^{N\times r}\mid Z^\top Z = I_r\}$ is the Stiefel manifold.
	 The solution of \eqref{eq_ssc} is then passed to the K-Means algorithm for the final clustering. SSC is an important example of the general situation of \eqref{problem_single} where $c(\cdot)$ is a nonlinear mapping.
	\item \textbf{Sparse principal component analysis (SPCA).} Principal Component Analysis (PCA) \citep{pearson1901liii} is one of the most fundamental statistical tools
	for analyzing high-dimensional data. SPCA seeks principal components with very few nonzero components. For a given data matrix $A\in\mathbb{R}^{N_1\times N}$ and $\kappa>0$, SPCA seeks the leading $r$ ($r\leq \min\{N_1,N\}$) sparse
	loading vectors and can be formulated as 
	\begin{equation}\label{eq_spca}
		\min_{U\in\St(N,r)} -\|AU\|_F^2 + \kappa \|\mbox{vec}(U)\|_1.
	\end{equation}
	Here, $\kappa>0$ is a weighting parameter. The operator $\mbox{vec}(\cdot)$ stacks the columns of $A\in \mathbb{R}^{m\times n}$ into a single column vector of size $mn\times 1$. SPCA is an important example of \eqref{problem_single}, where $c(\cdot)$ is the identity mapping.

\end{itemize}

\subsection{Existing Works and Challenges}

A variety of algorithms have been developed to address the problem \eqref{problem_single} or its simplified variants. These include Riemannian subgradient-type methods \citep{borckmans2014riemannian,hosseini2017riemannian,hosseini2018line,li2021weakly}, constraint-dissolving methods \citep{hu2023constraint}, Riemannian proximal methods \citep{chen2020proximal,huang2022riemannian,huang2023inexact,wang2022manifold,chen2024nonsmooth}, Riemannian smoothing-type algorithms \citep{beck2023dynamic,peng2023riemannian,zhang2023riemannian}, splitting-type methods \citep{lai2014splitting,kovnatsky2016madmm,li2023riemannian},Riemannian augmented Lagrangian methods \cite{deng2023manifold,zhou2023semismooth,deng2024oracle,zhou2022robust} and Riemannian min-max algorithms \citep{xu2023efficient2,xu2024riemannian}. Among these methods, \cite{xu2024oracle,zhou2023semismooth,zhou2022robust} used the Riemannian augmented Lagrangian method for solving \eqref{problem_single}, and \cite{wang2022manifold} solved the SSC problem \eqref{eq_ssc} by their proposed manifold proximal linear (ManPL) algorithm. Other algorithms mentioned above focus only on simplified cases of \eqref{problem_single}, where $c(\cdot)$ reduces to an identity or linear mapping.
 
Many of the algorithms mentioned above are double-loop algorithms that require iteratively solving subproblems. For double-loop algorithms designed for Euclidean space, adaptive stopping conditions can provide better control over the subproblem solving compared to fixed or predetermined stopping conditions and thus show better numerical performances (see e.g. \cite{zheng2023new,zheng2024adaptive}). However, the research on adaptive stopping conditions for nonsmooth manifold optimization remains limited. \cite{huang2022riemannian2,huang2023inexact} used adaptive stopping conditions, but their methods only applied to the simplified situation of $c(\cdot)$ being an identity mapping. Moreover, they only discussed the iteration complexities for the outer loops instead of the total number of subproblem iterations. For algorithms that apply to the general \eqref{problem_single},  \cite{xu2024oracle,zhou2023semismooth,zhou2022robust} used predetermined stopping conditions, and the theoretical analysis in \cite{wang2022manifold} assumed that they could exactly solve the subproblems, which is nearly impossible in practice.

In this paper, we consider the Riemannian proximal linear method for solving \eqref{problem_single}. We introduce some notation first. For any $y,z\in\mathbb{R}^n$ and $t>0$, we denote 
$$F_t(z;y) := F(z;y) + \frac{1}{2t}\|z - y\|_2^2,$$ 
where
$$\: F(z;y) := f(y)+\nabla f(y)^\top (z-y)+h(c(y)+\nabla c(y)(z-y)).$$
Here, $\nabla f(\cdot)\in\mathbb{R}^n$ denotes the gradient of $f$, and $\nabla c(\cdot)\in\mathbb{R}^{m\times n}$ denotes the Jacobian matrix of $c$. First, we discuss the proximal linear method for optimization on the Euclidean space, i.e., \eqref{problem_single} when $\M=\R^n$. 
The proximal linear method for solving this problem in Euclidean space iterates as 
\begin{equation}\label{eq_form_inexact_pl}
    z_{k+1}\approx \argmin_{z\in\mathbb{R}^n} F_{t_k}(z;z_k).
\end{equation}
 Here, $t_k>0$ is the step size, and ``$\approx$" means that the strongly convex function $F_{t_k}(\cdot;z_k)$ is minimized inexactly. \cite{drusvyatskiy2019efficiency} analyzed the proximal linear algorithm, inexactly solved the subproblems with first-order algorithms under predetermined stopping conditions, and proved that their algorithm can find an $\epsilon-$stationary point with $O(1/\epsilon^2)$ main iterations and $O(\frac{1}{\epsilon^3}\log \frac{1}{\epsilon})$ total subproblem iterations. \cite{zheng2023new,zheng2024adaptive} proposed the inexact proximal linear algorithms that inexactly solved the subproblem under adaptive subproblem stopping conditions. However, their analysis is limited to the robust phase retrieval problem featured by the sharpness assumption \citep{zheng2024smoothed}.

Next, we discuss the manifold proximal linear (ManPL) algorithm \citep{wang2022manifold}. Denote $\mathrm{T}_z \mathcal{M}$ as the tangent space at $z\in \mathcal{M}$ and $\mathrm{Retr}_z: \mathrm{T}_z \mathcal{M}\rightarrow \mathcal{M}$ as the retraction mapping. In the $k$-th iteration, ManPL inexactly solves the subproblem 
\begin{equation}\label{eq_imanpl}
	\Tilde{z}_{k+1}\approx \argmin_{z\in  z_k+\mathrm{T}_{z_k}\mathcal{M}} F_{t_k}(z;z_k),t_k>0,
\end{equation}
and then updates the iterate as:
\begin{equation}\label{eq_retraction}
	z_{k+1} = \Retr_{z_k}(\alpha_k(\Tilde{z}_{k+1} - z_k)),\alpha_k\in (0,1].
\end{equation}
Here, $t_k$ is the step size, and the shrinkage factor $\alpha_k$ is determined by Armijo backtracking line search. \cite{wang2022manifold} proved that ManPL finds an $\epsilon-$stationary point with $O(1/\epsilon^2)$ outer iterations when the subproblem \eqref{eq_imanpl} is solved exactly. 
Moreover, \cite{wang2022manifold} used the proximal point algorithm along with the adaptive regularized semi-smooth Newton method \citep{xiao2018regularized} to solve the subproblem \eqref{eq_imanpl}, which is inefficient due to the high cost of solving linear systems. As a result, when applied to the SSC problem in numerical experiments, ManPL takes more CPU time compared to candidate algorithms on relaxed optimization problems (see Tables 3 and 4 in \cite{wang2022manifold} for more details).

\subsection{Our Methods and Contributions}

Now, we introduce our inexact manifold proximal linear  (IManPL) algorithm. In the $k$-th iteration, our IManPL inexactly solves the subproblem \eqref{eq_imanpl} to find an inexact solution $\Tilde{z}_{k+1}$ such that 
\begin{equation}\label{eq_inexact_on_tangent}
	\Tilde{z}_{k+1}\in  \mathrm{T}_{z_k}\mathcal{M} + z_k
\end{equation}
using one of the following inexact termination conditions:
\begin{equation}\label{low-high0}
	F_{t_k}(\Tilde{z}_{k+1};z_k) - F_{t_k}(S_{t_k}(z_k);z_k)\leq \left\{\begin{array}{ll} 
		\rho_l\left(F(z_k) - F_{t_k}(\Tilde{z}_{k+1};z_k)\right), \ \rho_l>0, & \mbox{(LACC),}\\
		\frac{\rho_h}{2t_k}\|\Tilde{z}_{k+1} - z_k\|_2^2, \ \rho_h\in (0,1/4), & \mbox{(HACC)}.
        \end{array}\right.
\end{equation}
Here, $S_t(y) := \argmin_{z\in y+\mathrm{T}_y\mathcal{M}} F_t(z;y),$
$\rho_l>0$ and $\rho_h\in (0,1/4)$ are hyperparameters, and we call the first option low accuracy conditions (LACC) and the second option high accuracy conditions (HACC). As we discuss later in Lemma \ref{lemma_sufficiency_lh}, (HACC) implies (LACC) for some specific choices of $\rho_l,\rho_h$. Both options require that $\Tilde{z}_{k+1}\in z_k+\mathrm{T}_{z_k}\mathcal{M}$, i.e., the difference for the update lies in the tangent space. These conditions are motivated by the adaptive conditions in the Euclidean space \cite{zheng2023new,zheng2024adaptive}. For the retraction step \eqref{eq_retraction}, denoting 
$$c_0 :=\begin{cases}
	1+1/(\sqrt{1+\rho_l} + \sqrt{\rho_l})^2, & \mbox{if  
		(LACC) is used},\\
	1+1/\left(\sqrt{1+\rho_h/(1-2\sqrt{\rho_h})} + \sqrt{\rho_h/(1-2\sqrt{\rho_h})}\right)^2, & \mbox{if 
		(HACC) is used},
\end{cases}
$$
we use Armijo backtracking line search and let $\alpha_k$ be the largest value in $\{2^{-s}:s\in\mathbb{N}\}$ such that the following \tcb{condition} holds:
\begin{equation}\label{eq_retract_cond2}
	\frac{1}{2}\left(F(z_k) + F(z_k + \alpha_k(\Tilde{z}_{k+1} - z_k);z_k)\right)  - F(z_{k+1})\geq 0.
\end{equation}
A prototype of our IManPL algorithm for solving \eqref{problem_single} is described in Algorithm \ref{alg:IManPL-ls}. 
We will discuss the subproblem solver for \eqref{eq_imanpl} and guarantees on reaching \eqref{low-high0}-(LACC) or \eqref{low-high0}-(HACC) later in Section \ref{sec:subsolver_overall}.
\begin{algorithm}[ht]  
	\caption{IManPL --  A Prototype}
	\label{alg:IManPL-ls}
	\begin{algorithmic}
		\STATE \textbf{Input:} Initial point $z_0\in\mathcal{M}$, step sizes  $t_k>0$, parameter $\rho_l>0$ or $\rho_h\in (0,1/4)$, inexact type (IT) $=$ (LACC) or (HACC).
		\FOR{$k = 0, 1, \ldots, $}
		\STATE Find $\tilde{z}_{k+1}$ by solving \eqref{eq_imanpl} such that \eqref{low-high0}-(LACC) holds or \eqref{low-high0}-(HACC) holds, determined by (IT).
		\FOR{$s = 0, 1, \ldots, $}
		\STATE $\alpha_k^{\mbox{tem}} \leftarrow 2^{-s}$, $z_{k+1}^{\mbox{tem}} \leftarrow \Retr_{z_k}(\alpha_k^{\mbox{tem}}(\Tilde{z}_{k+1} - z_k))$.
		\IF{\eqref{eq_retract_cond2} \tcb{holds} with $\alpha_k\leftarrow \alpha_k^{\mbox{tem}}$ and $z_{k+1}\leftarrow z_{k+1}^{\mbox{tem}}$}
		\STATE Break.
		\ENDIF
		\STATE $\alpha_k\leftarrow \alpha_k^{\mbox{tem}}$, $z_{k+1}\leftarrow z_{k+1}^{\mbox{tem}}$.
		\ENDFOR
		\ENDFOR
	\end{algorithmic}
\end{algorithm}

Our contributions are summarized below. 
\begin{enumerate}[(i)]
	\item We propose the IManPL algorithm for solving \eqref{problem_single}, which uses adaptive stopping conditions when inexactly solving the subproblem \eqref{eq_imanpl}. To the best of our knowledge, this is the first adaptive algorithm with both low and high accuracy conditions for the nonsmooth manifold composite optimization \eqref{problem_single}. 
	\item Under some mild conditions, we prove that any \tcb{accumulation} point of the sequence generated by IManPL is a stationary point, and IManPL finds an $\epsilon-$stationary point in $O(1/\epsilon^2)$ main iterations. This rate matches the $O(1/\epsilon^2)$ rate for ManPL in \cite{wang2022manifold} that assumes the subproblem is solved exactly. It also matches the complexity of the inexact proximal linear (IPL) algorithm in \cite{zheng2023new}, and the proximal linear (PL) method in \cite{drusvyatskiy2019efficiency} in Euclidean space.
	\item We solve the dual subproblems using accelerated proximal gradient (APG) \citep{tseng2008accelerated} under the conditions in Section \ref{sec:subsolver_overall}. When $c(\cdot)$ is the identity mapping and the proximal mapping of $h$ is semismooth, we can also use the safeguarded adaptive regularized semi-smooth Newton method (ASSN) \citep{xiao2018regularized}, as detailed in Appendix \ref{subsec:assn}. Both subproblem solvers are equipped with verifiable stopping conditions that imply \eqref{low-high0}. For a fixed positive outer step size and the APG nondegeneracy condition in Section \ref{sec:subsolver_overall}, IManPL can find an $\epsilon-$stationary point with a total $O(1/\epsilon^3)$ iterations in \tcb{APG} for solving all the subproblems \eqref{eq_imanpl}, which gives the first-order oracle complexity. 
    The $O(1/\epsilon^3)$ bound matches the best-known first-order oracle complexity order, also achieved by augmented Lagrangian methods \citep{xu2024oracle,deng2024oracle}, while using adaptive primal-dual stopping tests. It avoids the logarithmic factor in the $O(\epsilon^{-3}\log(1/\epsilon))$ bound reported by \cite{drusvyatskiy2019efficiency} for the Euclidean prox-linear method. These comparisons concern complexity orders under the respective assumptions and approximate-stationarity measures. 
\end{enumerate}

Table \ref{tab:table_comp_candidate} summarizes the comparison of our IManPL with closely related works.
\begin{table}[ht]
	\centering
	\begin{tabular}{|c|c|c|c|c|c|c|}
		\hline
		Algorithm & $c(\cdot)$ & inexact & adaptive &  total & stationary \\
		\hline
		IRPG \citep{huang2023inexact} & identity & $\checkmark$ & $\checkmark$ & $\times$ & $\checkmark$ \\ 
		\hline
		AManPG \citep{huang2022riemannian2} & identity & $\checkmark$ & $\checkmark$ & $\times$ & $\checkmark$ \\ 
		\hline
		MAL \citep{zhou2023semismooth,zhou2022robust} & general & $\checkmark$ & $\times$ & $\times$ & $\checkmark$ \\ 
		\hline
		RiAL \citep{xu2024oracle} & general & $\checkmark$ & $\times$ & $\checkmark$ & $\checkmark$ \\ 
		\hline
		ManPL \citep{wang2022manifold} & general & $\times$ & $\times$ & $\times$ & $\checkmark$\\ 
		\hline
		\textbf{IManPL (ours)} & general & $\checkmark$ & $\checkmark$ & $\checkmark$ & $\checkmark$ \\ 
		\hline
	\end{tabular}
	\caption{Summary of algorithms for solving \eqref{problem_single}. ``inexact'' indicates whether the algorithm allows the subproblem to be solved inexactly. ``adaptive'' indicates whether the algorithm uses adaptive subproblem stopping conditions. ``total'' indicates whether the number of the total subproblem iterations is analyzed. ``stationary" indicates whether an asymptotic or approximate-stationarity guarantee is analyzed in the cited work. The approximate-stationarity definitions need not coincide across methods.}
	\label{tab:table_comp_candidate}
\end{table}

The rest of this paper is organized as follows. Section \ref{sec:prelimi} introduces preliminaries, notation, and Assumptions. Section \ref{sec:converge_main} provides the convergence analysis in terms of the main iteration. Section \ref{sec:subsolver_overall} provides the subproblem solver and the overall iteration complexity. Section \ref{sec:numerical} provides numerical experiments on the SSC problem. We draw some concluding remarks in Section \ref{sec:conclusion}.  
Appendix \ref{subsec:assn} describes the ASSN algorithm for solving the subproblem when $c(\cdot)$ is the identity mapping. Appendix \ref{sec:numerical_spca} provides numerical experimental results for SPCA. The code used for the numerical experiments reported in Section~\ref{sec:numerical} and Appendix~\ref{sec:numerical_spca} is available at \url{https://github.com/LucasXinYu/IManPL_repo}.

\section{Preliminaries, Notation, and Assumptions}\label{sec:prelimi}
We begin by introducing the notation and some concepts in Riemannian optimization \citep{absil2008optimization,boumal2023introduction}. 
Let $\langle \,\!\cdot\,, \cdot\rangle$ and $\| \cdot \| $ denote the standard inner product and its induced norm on the Euclidean space $\mathbb{R}^n$, respectively. 
$\mathcal{M}$ is a Riemannian manifold embedded in $\mathbb{R}^n$ and $ \mathrm{T}_{z}\mathcal{M} $ denotes the tangent space to $\mathcal{M}$ at $z\in\mathcal{M}$.
Throughout this paper, the Riemannian metric on $\mathcal{M}$ is induced from the standard Euclidean product. 
The Riemannian gradient of the smooth function $f: \mathbb{R}^n \rightarrow \mathbb{R}$ at a point $z \in \mathcal{M}$ is given by
$\rgrad f(z)=\proj_{\mathrm{T}_{z}\mathcal{M}}(\nabla f(z))$, where $\nabla f(z)$ is the Euclidean gradient of $f$ at $z$ and $\proj_{\mathrm{T}_z \mathcal{M}}(\cdot)$ is the Euclidean projection operator onto $\mathrm{T}_z \mathcal{M}$. 
A retraction at $z \in \mathcal{M}$ is a smooth mapping $\mathrm{Retr}_z: \mathrm{T}_z \mathcal{M} \to \mathcal{M}$ satisfying (i) $\mathrm{Retr}_z(\mathbf{0}_z) = z$, where $\mathbf{0}_z$ is the zero element in $\mathrm{T}_z \mathcal{M}$; (ii) $\frac{\mathrm{d}}{\mathrm{d} t} \mathrm{Retr}_z (t v)|_{t = 0} = v$ for all $v \in \mathrm{T}_z \mathcal{M}$. 

Throughout the paper, we assume the following assumptions hold for problem \eqref{problem_single}.
\begin{assumption}\label{ass:problem}
	\begin{enumerate}
		\item[(a)] $f$ is $L_f$-smooth, i.e., 
		$\|\nabla_{z} f(z) - \nabla_{z} f(z')\|_2\leq L_f\|z-z'\|_2,\ \forall z,z{'}\in\mathbb{R}^{n}.$
		\item[(b)] $h$ is convex and $L_{h}$-Lipschitz continuous, i.e., $|h(y) - h(y')|\leq L_{h}\|y - y'\|_2, \ \forall y,y'\in\mathbb{R}^{m}.$
		\item[(c)] The Jacobian of $c$ is $L_{c}$-Lipschitz continuous, i.e., $\|\nabla c(z) - \nabla c(z')\|_2\leq L_c\|z - z'\|_2,\ \forall z,z'\in\mathbb{R}^{n}.$
		\item[(d)] $\mathcal{M}$ is compact. 
	\end{enumerate}
\end{assumption}
Note that Assumption \ref{ass:problem} (d) implies that there exist positive constants $M_1$ and $M_2$ such that 
\begin{equation}\label{retraction_property}
\|\mathrm{Retr}_{z}(\xi) - z\|_2\leq M_1\|\xi\|_2,\|\mathrm{Retr}_{\tcb{z}}(\xi) - (z+\xi)\|_2\leq M_2\|\xi\|_2^2,\forall z\in \mathcal{M},\xi\in \mathrm{T}_z \mathcal{M}.
\end{equation}
See Appendix B of \cite{boumal2019global}.
We also use the following notation: 
\begin{equation}\label{def-L-GF}
    L = L_f + L_hL_c, G_F = \sup_{z\in\mathcal{M}} \|\nabla_{z} f(z)\|_2 + L_h\sup_{z\in\mathcal{M}} \|\nabla c(z)\|_2<\infty.
\end{equation}

Finally, we define the stationary point for \eqref{problem_single}.
\begin{definition}\label{stationary_point}
	A point $z\in\mathcal{M}$ is called a stationary point of problem \eqref{problem_single} if it satisfies the following first-order condition:
	\begin{equation}\label{opt-cond}
		0\in \Proj _{\mathrm{T}_{z}\mathcal{M}} \left(\nabla_{z}  f(z) + \nabla c(z)^\top  \partial h(c(z))\right).
	\end{equation}
    We call $z\in\mathcal{M}$ an $\epsilon-$stationary point of \eqref{problem_single} if
    $\|{(z - S_t(z))}/{t}\|_2\leq \epsilon.$
\end{definition}

\section{Convergence Analysis for Main Iterations}\label{sec:converge_main}

In this section, we prove two convergence results of Algorithm \ref{alg:IManPL-ls}. The first one is the iteration complexity of obtaining an $\epsilon$-stationary point, and the second one is the global convergence to a stationary point. We present some technical lemmas first. 

\subsection{Technical Lemmas}
\begin{lemma}[Model Approximation and Weak Convexity; cf. Lemmas 3.2 and 4.2 in \cite{drusvyatskiy2019efficiency}]\label{lemma:weak_convexity}
	For any $y\in\mathbb{R}^n$, we have
	\begin{equation}\label{weak_rel1}
		|f(y)+\nabla f(y)^\top (z-y) - f(z)|\leq \frac{L_f}{2}\|z-y\|_2^2,\forall z\in\mathbb{R}^n,
	\end{equation}
	\begin{equation}\label{weak_rel2}
		|h(c(y)+\nabla c(y)(z-y)) - h(c(z))|\leq \frac{L_hL_c}{2}\|z-y\|_2^2,\forall z\in\mathbb{R}^n,
	\end{equation}
	\begin{equation}\label{weak_rel3}
		|F(z;y) - F(z)|\leq \frac{L}{2}\|z - y\|_2^2,\forall z\in\mathbb{R}^n, 
	\end{equation}
	\begin{equation}\label{weak_rel4}
		F_t(z;y) \geq F(z),\forall z\in\mathbb{R}^n, 0<t\leq L^{-1}.
	\end{equation}
Moreover, $F$ is $L$-weakly convex, that is, $F(\cdot)+(L/2)\|\cdot\|_2^2$ is convex.
\end{lemma}
The approximation estimates follow from the Lipschitz continuity of $\nabla f$ and $\nabla c$ and the convexity and Lipschitz continuity of $h$. The weak convexity assertion combines the $L_f$-weak convexity of $f$ with the $L_hL_c$-weak convexity of $h\circ c$.

\begin{lemma}\label{strong_convex_Ft}
	For any $z\in\mathcal{M},\tilde{z}\in \mathrm{T}_z\mathcal{M} + z$ and $t>0$, we have
	$$F_t(\tilde{z};z) - F_t(S_t(z);z)\geq \frac{1}{2t}\|\tilde{z}-S_t(z)\|_2^2.$$
\end{lemma}
\begin{proof}[Proof of Lemma \ref{strong_convex_Ft}]
	This holds from the $1/t-$strong convexity of $F_t(\cdot;z)$ on $\mathrm{T}_z\mathcal{M} + z$.
\end{proof}

Next, we discuss the inexactness of solving \eqref{eq_imanpl}. 
For any $z\in\mathcal{M},\Tilde{z}\in z + \mathrm{T}_z\mathcal{M}$ and $t>0$, we use 
$$\varepsilon_{t}(\Tilde{z};z) = F_{t}(\Tilde{z};z) - F_{t}(S_{t}(z);z)$$
to measure the subproblem accuracy. 
\begin{lemma}\label{lemma_relationship_quantities}
	For any $z\in\mathcal{M},\Tilde{z}\in z + \mathrm{T}_z\mathcal{M}$, $t>0$ and $\rho\in (0,1)$, the following inequalities hold:
	\begin{equation}\label{rel1}
		\frac{1}{2t}\|z - \Tilde{z}\|_2^2\leq \frac{1}{2t}\frac{\|z - S_{t}(z)\|_2^2}{\rho} + \frac{\varepsilon_{t}(\Tilde{z};z)}{1-\rho},
	\end{equation}
	\begin{equation}\label{rel2}
		\frac{1}{2t}\|z - S_{t}(z)\|_2^2\leq \frac{1}{2t}\frac{\|z - \Tilde{z}\|_2^2}{\rho} + \frac{\varepsilon_{t}(\Tilde{z};z)}{1-\rho}.
	\end{equation}
\end{lemma}
\begin{proof}[Proof of Lemma \ref{lemma_relationship_quantities}]
	These two relationships follow from the Cauchy-Schwarz inequality and the fact that $\varepsilon_{t}(\Tilde{z};z)\geq \frac{1}{2t}\|\Tilde{z} - S_{t}(z)\|_2^2$ that follows from Lemma \ref{strong_convex_Ft}.
\end{proof}

Next, we show that \eqref{low-high0}-(HACC) implies \eqref{low-high0}-(LACC) for some specific choices of $\rho_l,\rho_h$.
\begin{lemma}\label{lemma_sufficiency_lh}
	For any $z\in\mathcal{M}$, $\Tilde{z}\in \mathrm{T}_{z}\mathcal{M} + z$, and $t>0$, if 
\begin{equation}\label{high_temp}
        F_{t}(\Tilde{z};z) - F_{t}(S_{t}(z);z)\leq \frac{\rho_h}{2t}\|\Tilde{z} - z\|_2^2, \quad \rho_h\in (0,1/4),
    \end{equation}
    then we have
    $$F_{t}(\Tilde{z};z) - F_{t}(S_{t}(z);z)\leq \rho_l\left(F(z) - F_{t}(\Tilde{z};z)\right), \quad \rho_l = \frac{\rho_h}{1-2\sqrt{\rho_h}}.$$
\end{lemma}
\begin{proof}[Proof of Lemma \ref{lemma_sufficiency_lh}]
	We have
	$F_t(z;z) - F_t(S_t(z);z) = F_{t}(z;z) - F_{t}(\Tilde{z};z) + \varepsilon_{t}(\Tilde{z};z)\geq \frac{1}{2t}\|z-S_{t}(z)\|_2^2\geq \frac{\rho}{2t}\|\Tilde{z}-z\|_2^2-\frac{\rho}{1-\rho}\varepsilon_{t}(\Tilde{z};z)$ with $\rho = 1-\sqrt{\rho_h}$. Here, the first inequality is from Lemma \ref{strong_convex_Ft}, and the second one is from \eqref{rel1} of Lemma \ref{lemma_relationship_quantities}. Thus, $F_{t}(z;z) - F_{t}(\Tilde{z};z) + \frac{1}{1-\rho}\varepsilon_{t}(\Tilde{z};z)\geq \frac{\rho}{2t}\|\Tilde{z}-z\|_2^2$.
	Applying \eqref{high_temp} to $\varepsilon_{t}(\Tilde{z};z)$, we get $F_{t}(z;z) - F_{t}(\Tilde{z};z)\geq \frac{1-2\sqrt{\rho_h}}{\rho_h}\frac{\rho_h}{2t}\|z - \Tilde{z}\|_2^2$. Applying \eqref{high_temp} again to $\frac{\rho_h}{2t}\|z - \Tilde{z}\|_2^2$, we have
	\begin{equation}\label{sufficient_low_high}
		\varepsilon_{t}(\Tilde{z};z)\leq \frac{\rho_h}{1-2\sqrt{\rho_h}}\left(F_{t}(z;z) - F_{t}(\Tilde{z};z)\right).
	\end{equation}
	This finishes the proof.
\end{proof}
Setting $(z,\tilde{z},t)$ in Lemma \ref{lemma_sufficiency_lh} to $(z_k,\tilde{z}_{k+1},t_k)$ in Algorithm \ref{alg:IManPL-ls} shows that \eqref{low-high0}-(HACC) implies \eqref{low-high0}-(LACC) for some specific choices of $\rho_l,\rho_h$. \tcb{Thus, (HACC) shows higher accuracy for solving the subproblem \eqref{eq_imanpl} than (LACC). This explains the words ``low" and ``high" in their names, which were originally used in \cite{zheng2023new} and \cite{zheng2024adaptive}. However, the higher accuracy of (HACC) implies a higher complexity for solving \eqref{eq_imanpl}. This trade-off makes it difficult to directly tell which condition is better in terms of the overall complexity for solving \eqref{problem_single}.}

\subsection{Iteration Complexity of Obtaining an \texorpdfstring{$\epsilon$}{epsilon}-stationary point}\label{sec:complexity}
For Algorithm \ref{alg:IManPL-ls}, we can prove the following lemma for sufficient decrease of the objective function.
\begin{lemma}[Sufficient Decrease]\label{lemma_descent_ls}
    	\begin{itemize}
	\item[(a)] When \eqref{low-high0}-(LACC) holds, we have
	$$F(z_k) - F(z_{k+1}) \geq c_{k,l}\|z_k - S_{t_k}(z_k)\|_2^2.$$
	
	\item[(b)] When \eqref{low-high0}-(HACC) holds, we have
	$$F(z_k) - F(z_{k+1}) \geq c_{k,h}\|z_k - S_{t_k}(z_k)\|_2^2, \quad .$$
	\end{itemize}
	Here, $c_{k,l},c_{k,h}$ are positive scalars that will be specified in the proof. Specifically, $c_{k,l}$ depends on $t_k$ and $\rho_l$, and $c_{k,h}$ depends on $t_k$ and $\rho_h$.
\end{lemma}

\begin{proof}[Proof of Lemma \ref{lemma_descent_ls}]
	We first provide the proof based on \eqref{low-high0}-(LACC). For $\hat{\alpha}_k\in [0,1]$, we denote 
	$$\hat{z}_{k+1} = z_k + \hat{\alpha}_k(\Tilde{z}_{k+1} - z_k), \quad \overline{z}_{k+1} = \mathrm{Retr}_{z_k}(\hat{z}_{k+1} - z_k), \quad \tau = (\max\{M_1^2,1\}L)^{-1}.$$
	Lemma \ref{strong_convex_Ft} indicates that
	$$F_{t_k}(z_k;z_k) - F_{t_k}(\Tilde{z}_{k+1};z_k) +\varepsilon_{t_k}(\Tilde{z}_{k+1};z_k)\geq \frac{1}{2t_k}\|z_k - S_{t_k}(z_k)\|_2^2.$$
	Combined with \eqref{low-high0}-(LACC), we have
	\begin{equation}\label{eq0_proof_sufficient_decrease}
		F_{t_k}(z_k;z_k) - F_{t_k}(\Tilde{z}_{k+1};z_k) \geq \frac{1}{2{t_k}(1+\rho_l)}\|z_k - S_{t_k}(z_k)\|_2^2.
	\end{equation}
	This indicates that $F(z_k;z_k) - F(\Tilde{z}_{k+1};z_k)\geq \frac{1}{2{t_k}(1+\rho_l)}\|z_k - S_{t_k}(z_k)\|_2^2$. From the convexity of $F(\cdot;z_k)$, we also have
	\begin{equation}\label{eq0_proof_descent_ls}
		F(z_k;z_k) - F(\hat{z}_{k+1};z_k)\geq \frac{\hat{\alpha}_k}{2{t_k}(1+\rho_l)}\|z_k - S_{t_k}(z_k)\|_2^2.
	\end{equation}
	In addition, by \eqref{eq0_proof_sufficient_decrease} and \eqref{rel1} in Lemma \ref{lemma_relationship_quantities} with $(z,\tilde{z},t) = (z_k,\tilde
    {z}_{k+1},t_k)$ and $\rho\in (0,1)$ such that $(1-\rho)/\rho = \sqrt{\rho_l/(1+\rho_l)}$, we have that
    \begin{equation}\label{temp_proof_suffdecrease}
        F_{t_k}(z_k;z_k) - F_{t_k}(\Tilde{z}_{k+1};z_k)\geq \frac{\rho}{2{t_k}(1+\rho_l)}\|z_k - \tilde{z}_{k+1}\|_2^2 - \frac{\rho}{(1-\rho)(1+\rho_l)}\varepsilon_{t_k}(\Tilde{z}_{k+1};z_k),
    \end{equation}
    which further yields
	\begin{align*}
		&\frac{1-\rho+\rho_l}{(1-\rho)(1+\rho_l)}\left(F_{t_k}(z_k;z_k) - F_{t_k}(\Tilde{z}_{k+1};z_k)\right) \\
        = & \left(1+\frac{\rho\rho_l}{(1-\rho)(1+\rho_l)}\right)\left(F_{t_k}(z_k;z_k) - F_{t_k}(\Tilde{z}_{k+1};z_k)\right)\\
		\geq & F_{t_k}(z_k;z_k) - F_{t_k}(\Tilde{z}_{k+1};z_k) + \frac{\rho}{(1-\rho)(1+\rho_l)}\varepsilon_{t_k}(\Tilde{z}_{k+1};z_k) \\
        \geq & \frac{\rho}{2{t_k}(1+\rho_l)}\|z_k - \tilde{z}_{k+1}\|_2^2,
	\end{align*} 
where the first inequality follows from \eqref{low-high0}-(LACC), and the second inequality follows from \eqref{temp_proof_suffdecrease}. 
Thus,
	$$F_{t_k}(z_k;z_k) - F_{t_k}(\Tilde{z}_{k+1};z_k)\geq \frac{\rho}{2{t_k}(1+\frac{\rho_l}{1-\rho})}\|z_k - \Tilde{z}_{k+1}\|_2^2 = \frac{1}{2{t_k}(\sqrt{1+\rho_l} + \sqrt{\rho_l})^2}\|z_k - \Tilde{z}_{k+1}\|_2^2,$$
	which also indicates that
	$$F(z_k;z_k) - F(\Tilde{z}_{k+1};z_k)\geq \left(\frac{1}{2{t_k}} + \frac{1}{2{t_k}(\sqrt{1+\rho_l} + \sqrt{\rho_l})^2}\right)\|z_k - \Tilde{z}_{k+1}\|_2^2.$$
	This inequality, together with the convexity of $F(\cdot;z_k)$, further yields, 	\begin{equation}\label{eqneg1_proof_sufficient_decrease}
		F(z_k;z_k) - F(\hat{z}_{k+1};z_k)\geq \left(\frac{\hat{\alpha}_k}{2{t_k}} + \frac{\hat{\alpha}_k}{2{t_k}(\sqrt{1+\rho_l} + \sqrt{\rho_l})^2}\right)\|z_k - \Tilde{z}_{k+1}\|_2^2,
	\end{equation}
	which indicates that
	$$F(z_k) - F_{\tau}(\hat{z}_{k+1};z_k)\geq \left(\frac{\hat{\alpha}_k}{2{t_k}} + \frac{\hat{\alpha}_k}{2{t_k}(\sqrt{1+\rho_l} + \sqrt{\rho_l})^2} - \frac{\hat{\alpha}_k^2}{2\tau}\right)\|z_k - \Tilde{z}_{k+1}\|_2^2.$$
	
	Next, we establish a lower bound for $F_\tau(\hat{z}_{k+1};z_k) - F(\overline{z}_{k+1})$. We know that 
	\begin{align}\label{eq_decomp}
     & F_\tau(\hat{z}_{k+1};z_k) - F(\overline{z}_{k+1}) \\ = & 
	f(z_k) + \nabla f(z_k)^\top (\hat{z}_{k+1} - z_k)-f(\overline{z}_{k+1})+ h(c(z_{k}) + \nabla c(z_k)(\hat{z}_{k+1} - z_k)) - h(c(\overline{z}_{k+1})) \nonumber\\ & + \frac{1}{2\tau}\|z_k - \hat{z}_{k+1}\|_2^2.\nonumber
    \end{align}
	By \eqref{weak_rel1} in Lemma \ref{lemma:weak_convexity},
	$$f(z_k) - f(\overline{z}_{k+1})\geq \nabla f(z_k)^\top (z_k - \overline{z}_{k+1}) - \frac{L_f}{2}\|z_k - \overline{z}_{k+1}\|_2^2,$$
	which indicates that
	\begin{equation}\label{eq1_proof_sufficient_decrease}
		f(z_k) + \nabla f(z_k)^\top (\hat{z}_{k+1} - z_k)-f(\overline{z}_{k+1})\geq -\left(\sup_{z\in\mathcal{M}}\|\nabla f(z)\|_2\right)\|\hat{z}_{k+1} - \overline{z}_{k+1}\|_2 - \frac{L_f}{2}\|\overline{z}_{k+1} - z_k\|_2^2.
	\end{equation}
	Moreover, 
    \begin{align*}
	 & h(c(z_{k}) + \nabla c(z_k)(\hat{z}_{k+1} - z_k)) - h(c(\overline{z}_{k+1})) \\
    = & h(c(z_{k}) + \nabla c(z_k)(\hat{z}_{k+1} - z_k)) - h(c(z_{k})+ \nabla c(z_k)(\overline{z}_{k+1} - z_k))
	\\
    & + h(c(z_{k}) + \nabla c(z_k)(\overline{z}_{k+1} - z_k)) - h(c(\overline{z}_{k+1})).
    \end{align*}
	By Assumption \ref{ass:problem} (b) and (c), we have
	\begin{align*}
	    & h(c(z_{k}) + \nabla c(z_k)(\hat{z}_{k+1} - z_k)) - h(c(z_{k})+ \nabla c(z_k)(\overline{z}_{k+1} - z_k)) \\ \geq & -L_h\|\nabla c(z_k) (\hat{z}_{k+1} - \overline{z}_{k+1})\|_2 \\ 
        \geq & -L_h(\sup_{z\in\mathcal{M}} \|\nabla c(z)\|_2)\|\overline{z}_{k+1} - \hat{z}_{k+1}\|_2. 
	\end{align*}
        By \eqref{weak_rel2} in Lemma \ref{lemma:weak_convexity},
	$h(c(z_{k}) + \nabla c(z_k)(\overline{z}_{k+1} - z_k)) - h(c(\overline{z}_{k+1}))\geq -\frac{L_hL_c}{2}\|z_k - \overline{z}_{k+1}\|_2^2.$ Thus,
	\begin{align}
		& h(c(z_{k}) + \nabla c(z_k)(\hat{z}_{k+1} - z_k)) - h(c(\overline{z}_{k+1})) \nonumber \\
		\geq & -\frac{L_hL_c}{2}\|z_k - \overline{z}_{k+1}\|_2^2-L_h(\sup_{z\in\mathcal{M}} \|\nabla c(z)\|_2)\|\overline{z}_{k+1} - \hat{z}_{k+1}\|_2. \label{eq2_proof_sufficient_decrease}
	\end{align}
	By \eqref{eq_decomp}, \eqref{eq1_proof_sufficient_decrease} and \eqref{eq2_proof_sufficient_decrease}, we have (note that $L$ and $G_F$ are defined in \eqref{def-L-GF}):
	$$F_\tau(\hat{z}_{k+1};z_k) - F(\overline{z}_{k+1})\geq
	\frac{1}{2\tau}\|\hat{z}_{k+1} - z_k\|_2^2- \frac{L}{2}\|z_k - \overline{z}_{k+1}\|_2^2- G_F\|\overline{z}_{k+1} - \hat{z}_{k+1}\|_2.$$
	Noticing that $\frac{1}{2\tau}\|\hat{z}_{k+1} - z_k\|_{\tcb{2}}^2- \frac{L}{2}\|z_k - \overline{z}_{k+1}\|_2^2\geq 0$ 
	and 
	$G_F\|\overline{z}_{k+1} - \hat{z}_{k+1}\|_2\leq G_FM_2\|z_k - \hat{z}_{k+1}\|_2^2$
	from \eqref{retraction_property}, we further have
	$$F_\tau(\hat{z}_{k+1};z_k) - F(\overline{z}_{k+1})\geq 
	-G_FM_2\|z_k - \hat{z}_{k+1}\|_2^2 = 
	-\hat{\alpha}_k^2G_FM_2\|\Tilde{z}_{k+1} - z_k\|_2^2.$$
	Thus,
    \begin{equation}\label{temp_eq_alpha_k_sq}
        F(\hat{z}_{k+1};z_k) - F(\overline{z}_{k+1})\geq -\hat{\alpha}_k^2(G_FM_2 + 1/(2\tau))\|\Tilde{z}_{k+1} - z_k\|_2^2.
    \end{equation}
    Denote
	$$c_1 = 1 + \frac{1}{(\sqrt{1+\rho_l} + \sqrt{\rho_l})^2}, \quad c_{2,k} = t_k/\tau + 2t_kG_FM_2.$$
	Combining \eqref{eqneg1_proof_sufficient_decrease} and \eqref{temp_eq_alpha_k_sq} yields
	\begin{align}\label{eqneg1_proof_temp2} & \frac{1}{2}\left(F(z_k) + F(\hat{z}_{k+1};z_k)\right)  - F(\overline{z}_{k+1}) \\ = & \frac{1}{2}\left(F(z_k) - F(\hat{z}_{k+1};z_k)\right) + (F(\hat{z}_{k+1};z_k) - F(\overline{z}_{k+1})) \nonumber\\
	\geq & \frac{c_1\hat{\alpha}_k/2 - c_{2,k}\hat{\alpha}_k^2}{2t_k}\|z_k - \Tilde{z}_{k+1}\|_2^2.\nonumber
	\end{align}
    \tcb{From} \eqref{eqneg1_proof_temp2} we know that, 
    for any $\hat{\alpha}_k\in [0,c_{3,k}]$ with $c_{3,k} = \min\{1,c_1/(2c_{2,k})\}$, the following \tcb{inequality holds} 
	\begin{equation}\label{eq2_term_cond}
		\frac{1}{2}\left(F(z_k) + F(\hat{z}_{k+1};z_k)\right)  - F(\overline{z}_{k+1})\geq 0.
	\end{equation}
    \eqref{eq2_term_cond} indicate\tcb{s} that there must exist $\alpha_k\in [c_{3,k}/2,1]$ such that the line search condition 
    \eqref{eq_retract_cond2} is satisfied 
    with $z_{k+1} = \Retr_{z_k}(\alpha_k(\tilde{z}_{k+1}-z_k))$. Thus,
	\begin{align*}
	& F(z_k) - F(z_{k+1}) \\
	= & \frac{1}{2}\left(F(z_k) - F(z_k + \alpha_k(\Tilde{z}_{k+1} - z_k);z_k)\right) + \frac{1}{2}\left(F(z_k) + F(z_k + \alpha_k(\Tilde{z}_{k+1} - z_k);z_k)\right) - F(z_{k+1}) \\
	\geq & \frac{1}{2}\left(F(z_k) - F(z_k + \alpha_k(\Tilde{z}_{k+1} - z_k);z_k)\right) \\
	\geq & \frac{c_{3,k}}{8t_k(1+\rho_l)}\|z_k - S_{t_k}(z_k)\|_2^2,
	\end{align*}
	where, the first inequality is from \eqref{eq_retract_cond2}, and the second inequality is from \eqref{eq0_proof_descent_ls} with $\hat{\alpha}_k = \alpha_k\geq c_{3,k}/2$. This finishes the proof for part (a) by letting $c_{k,l} = \frac{c_{3,k}}{8t_k(1+\rho_l)}$.
	
	For the proof of (b), based on Lemma \ref{lemma_sufficiency_lh}, we only need to replace $\rho_l$ with $\rho' := \rho_h/(1-2\sqrt{\rho_h})$. More specifically, the same proof goes through by replacing $c_1$ with $c_1' = 1+ 1/(\sqrt{1+\rho'}+\sqrt{\rho'})^2$ and replacing $c_{3,k}$ with $c_{3,k}' = \min\{1,c_1'/(2c_{2,k})\}$. The desired result follows by letting $c_{k,h} = c_{3,k}'/(8t_k(1+\rho'))$.
\end{proof}

Now we are ready to provide the main iteration complexity of finding an $\epsilon-$stationary point.
\begin{theorem}\label{thm_main_iteration_complexity}
	Algorithm \ref{alg:IManPL-ls} with $t_k = t>0,\forall k\in\mathbb{N}$ finds an $\epsilon-$stationary point among the first
	$1+\left\lfloor \frac{F_0 - F_\star}{\beta_0\epsilon^2}\right\rfloor$ iterates $z_0,z_1,\ldots$; equivalently, at most $\left\lfloor (F_0-F_\star)/(\beta_0\epsilon^2)\right\rfloor$ outer updates are needed.  
	Here, $F_0 = F(z_0), F_\star = \inf_{z \in \mathcal{M}} F(z)$, $\beta_0 = t^2c_l$ if $IT=LACC$ and $\beta_0 = t^2c_h$ if $IT=HACC$. Here, under the fixed step size $t$, $c_l := c_{k,l},c_h:= c_{k,h}, \forall k\in\mathbb{N}$.
\end{theorem}
\begin{proof}[Proof of Theorem \ref{thm_main_iteration_complexity}]
Suppose that the $K$ inspected iterates $z_0,\ldots,z_{K-1}$ all fail the $\epsilon$-stationarity criterion. Then $\|z_k-S_t(z_k)\|_2/t>\epsilon$ for $k=0,\ldots,K-1$, and Lemma \ref{lemma_descent_ls} gives
\[
F(z_k)-F(z_{k+1})>\beta_0\epsilon^2,\qquad k=0,\ldots,K-1.
\]
Summing over these $K$ outer updates yields
\[
F_0-F_\star\geq\sum_{k=0}^{K-1}\bigl(F(z_k)-F(z_{k+1})\bigr)>K\beta_0\epsilon^2.
\]
Taking $K=1+\lfloor(F_0-F_\star)/(\beta_0\epsilon^2)\rfloor$ gives a contradiction. Hence at least one of $z_0,\ldots,z_{K-1}$ is $\epsilon$-stationary; generating it requires at most $K-1$ outer updates.
\end{proof}

\subsection{\tcb{Stationarity of Accumulation Points}}
In this subsection, we prove that \tcb{any accumulation point of} $\{z_k\}_{k=0}^\infty$ generated by Algorithm \ref{alg:IManPL-ls} is a stationary point of \eqref{problem_single}.

\begin{theorem}\label{thm_converge_to_stationary}
	For Algorithm \ref{alg:IManPL-ls} with $t_k = t>0, \forall k\in\mathbb{N}$ and IT being either (LACC) or (HACC), any accumulation point of $\{z_k\}_{k=0}^\infty$ is a stationary point of \eqref{problem_single}.
\end{theorem}

\begin{proof}[Proof of Theorem \ref{thm_converge_to_stationary}]
    First, accumulation points of $\{z_k\}_{k=0}^\infty$ exist because of the compactness of $\mathcal{M}$ in Assumption \ref{ass:problem}(d).
    We will use the notation $F_0,F_\star$ and $\beta_0$ defined in Theorem \ref{thm_main_iteration_complexity}. 
    Following the idea of \cite{drusvyatskiy2019efficiency}, we denote
	\begin{equation}\label{global-conv-def-tilde-S}
	\tilde{S}_{t'}(z_k) = \argmin_{z\in \mathrm{T}_{z_k}\mathcal{M} + z_k} F(z) + \frac{1}{2t'}\|z - z_k\|_2^2, \quad t' = t/(1+tL).
    \end{equation}

    The Lipschitz continuity in Assumption \ref{ass:problem}(a)(b)(c) and the compactness in Assumption \ref{ass:problem}(d) imply that $F_\star > -\infty$. Under the fixed step size $t$, by Lemma \ref{lemma_descent_ls}, we have 
    $F(z_k) - F(z_{k+1})\geq \beta_0\|z_k - S_t(z_k)\|_2^2/t^2.$ Thus, $\beta_0\sum_{k=0}^\infty\|z_k - S_t(z_k)\|_2^2/t^2\leq \sum_{k=0}^\infty \left(F(z_k) - F(z_{k+1})\right)\leq F_0 - F_\star<\infty$, which further indicates $\lim_{k\rightarrow \infty}\|z_k - S_t(z_k)\|_2 = 0$. For each fixed $k$, restrict $F$ and its model to the affine space $z_k+\mathrm{T}_{z_k}\mathcal{M}$ and use an orthonormal basis of the tangent space as Euclidean coordinates. Lemma \ref{lemma:weak_convexity} gives $L$-weak convexity and the same quadratic model-error bound in those coordinates. Therefore, the argument of Lemma 4.3 and Theorem 4.5 in \cite{drusvyatskiy2019efficiency} applies on that affine space, with constants depending only on $t$ and $L$, and gives
	\begin{equation}\label{eq_related_pp_pl}
		\left\|\frac{z_k - S_t(z_k)}{t}\right\|_2 = \Theta\left(\left\|\frac{z_k - \tilde{S}_{t'}(z_k)}{t'}\right\|_2\right),\forall k\in\mathbb{N},
	\end{equation}
	where $\Theta$ hides positive constant factors related to $t$ and $L$. Since $\lim_{k\rightarrow \infty} \left\|\frac{z_k - S_t(z_k)}{t}\right\|_2 = 0$, we have that
	\begin{equation}\label{eq_lim0_moreau}
		\lim_{k\rightarrow \infty} \left\|\frac{z_k - \tilde{S}_{t'}(z_k)}{t'}\right\|_2 = 0.
	\end{equation}
	According to Theorem 3.1 in \cite{drusvyatskiy2019efficiency},
    \begin{equation}\label{subgrad_expression}
        \partial F(z) = \nabla f(z) + [\nabla c(z)]^\top \partial h(c(z)), \quad \forall z\in\mathbb{R}^n.
    \end{equation}
	The first-order optimality condition for \eqref{global-conv-def-tilde-S} indicates that  there exists $\xi_k\in [\mathrm{T}_{z_k}\mathcal{M}]^\perp$ such that 
	\begin{equation}\label{env_manpl}
		\frac{1}{t'}(z_k - \tilde{S}_{t'}(z_k)) + \xi_k\in \partial F(\tilde{S}_{t'}(z_k)).
	\end{equation}
    Here $[\mathrm{T}_{z}\mathcal{M}]^\perp = \{\xi\mid \xi^\top v=0, \forall v\in \mathrm{T}_{z}\mathcal{M}\}$. We now prove that $\{\xi_k\}$ is bounded. 
	\eqref{eq_lim0_moreau} implies that $\sup_{k\in\mathbb{N}}\|z_k - \tilde{S}_{t'}(z_k)\|_2<\infty$.
    Together with the fact that $\{z_k: k\in\mathbb{N}\}$ is a bounded set because of the compactness of $\mathcal{M}$ from Assumption \ref{ass:problem}(d), we have that $\{\tilde{S}_{t'}(z_k): k\in\mathbb{N}\}$ is a bounded set. The boundedness, the Lipschitz continuity in Assumption \ref{ass:problem}(a)(b)(c) and \eqref{subgrad_expression} imply
    \begin{equation}\label{bound_subgrad}
        \sup_{k\in\mathbb{N}} \sup_{v\in \partial F(\tilde{S}_{t'}(z_k))}\|v\|_2<\infty.
    \end{equation}
     In addition, in \eqref{env_manpl}, $\frac{1}{t'}(z_k - \tilde{S}_{t'}(z_k))^\top \xi_k = 0$ because $\xi_k \in [\mathrm{T}_{z_k}\mathcal{M}]^\perp$ and $\frac{1}{t'}(z_k - \tilde{S}_{t'}(z_k))\in \mathrm{T}_{z_k}\mathcal{M}$ by \eqref{global-conv-def-tilde-S}. This implies
    \begin{equation}\label{bound_xi}
        \|\xi_k\|_2\leq \|\xi_k + \frac{1}{t'}(z_k - \tilde{S}_{t'}(z_k))\|_2,\forall k\in\mathbb{N}.
    \end{equation}
    \eqref{env_manpl}, \eqref{bound_subgrad} and \eqref{bound_xi} imply
\begin{equation}\label{boundness_orth}
\sup_{k\in\mathbb{N}}\|\xi_k\|_2<\infty.
    \end{equation} 
For any subsequence $\{z_{k^s}\}_{s=0}^\infty$ of $\{z_k\}_{k=0}^\infty$ that converges to $z_\star$, we consider the sequence $\{(z_{k^s},\xi_{k^s})\}_{s=0}^\infty$. By \eqref{boundness_orth}, there also exists a subsequence $\{(z_{k(r)},\xi_{k(r)})\}_{r = 0}^\infty$ of $\{(z_{k^s},\xi_{k^s})\}_{s=0}^\infty$ that converges to $(z_\star,\xi_\star)$.

The compactness of $\mathcal{M}$ and \eqref{eq_lim0_moreau} \tcb{imply} $\lim_{r\rightarrow \infty} z_{k(r)} = \lim_{r\rightarrow \infty} \tilde{S}_{t'}(z_{k(r)}) = z_\star\in\mathcal{M}$. 
Since $F$ is continuous, $F(\tilde{S}_{t'}(z_{k(r)}))\to F(z_\star)$. Thus \eqref{env_manpl} and the closedness property in Remark 1(ii) of \cite{bolte2014proximal} imply that
	\begin{equation}\label{eq_limit_rel0}
		\xi_\star\in\partial F(z_\star).
	\end{equation}
	Since the Euclidean projection $\mbox{Proj}_{\mathrm{T}_z\mathcal{M}}(\cdot)$ is smooth with respect to $z$ and $\xi_{k} = \mbox{Proj}_{[\mathrm{T}_{z_{k}}\mathcal{M}]^\perp}(\xi_{k})$, we have that 
	$\xi_\star = \lim_{r\rightarrow\infty}\xi_{k(r)} = \lim_{r\rightarrow\infty}\mbox{Proj}_{[\mathrm{T}_{z_{k(r)}}\mathcal{M}]^\perp}(\xi_{k(r)}) = \mbox{Proj}_{[\mathrm{T}_{z_{\star}}\mathcal{M}]^\perp}(\xi_{\star})$. Thus, \eqref{eq_limit_rel0} implies that $0\in\mbox{Proj}_{\mathrm{T}_{z_{\star}} \mathcal{M}}(\partial F(z_\star))$.
	This finishes the proof.
\end{proof}

\section{Subproblem Solver and Overall First-Order Oracle Complexity}\label{sec:subsolver_overall}

Note that the subproblem \eqref{eq_imanpl} is a structured convex problem, \tcb{and we assume in this section that the proximal operator of $h_\star$ (the convex conjugate of $h$) can be efficiently evaluated}. For the general smooth mapping $c(\cdot)$, we will use a first-order algorithm to approximately solve it such that \eqref{low-high0}-(LACC) or \eqref{low-high0}-(HACC) holds. Therefore, it is also important to analyze the oracle complexity needed for this first-order algorithm. Combining it with the main iteration complexity in Section \ref{sec:complexity}, we can analyze the first-order oracle complexity of IManPL, which is defined as the total number of subproblem iterations for solving all the subproblems \eqref{eq_imanpl}. When $c(\cdot)$ is the identity mapping, we find that solving \eqref{eq_imanpl} with a second-order algorithm might show better efficiency. Thus, the rest of this section uses the accelerated proximal gradient (Algorithm 1 in \cite{tseng2008accelerated}) to approximately solve the dual problem of \eqref{eq_imanpl} under the general $c(\cdot)$ such that \eqref{low-high0}-(LACC) or \eqref{low-high0}-(HACC) holds. Section \ref{sec:numerical} applies it to the sparse spectral clustering problem. Appendix \ref{subsec:assn} uses the adaptive semi-smooth Newton method (ASSN) \citep{xiao2018regularized} to solve \eqref{eq_imanpl} when $c(\cdot)$ is the identity mapping. Appendix \ref{sec:numerical_spca} applies it to the sparse principal component analysis.

\subsection{Subproblem solver and stopping criteria}\label{subsec:suff_stopping}
For the ease of discussion, we introduce some notation. In particular, we denote
$x = z - z_k\in\mathbb{R}^n, B_k = \nabla c(z_k)\in\mathbb{R}^{m\times n},c_k = \nabla f(z_k)\in\mathbb{R}^n, d_k = c(z_k)\in\mathbb{R}^m$ and 
$$\mathrm{T}_{z_k}\mathcal{M} = \{x\in\mathbb{R}^n:C_kx = 0\},$$
where $C_k\in\mathbb{R}^{\hat{n}\times n}$, $0\leq \hat{n} \leq n$ and $C_kC_k^\top  = I_{\hat{n}}$. Moreover, we denote $\hat{C}_k\in\mathbb{R}^{(n-\hat{n})\times n}$ as the orthogonal complement matrix for $C_k$, which means that $C_k\hat{C}_k^\top  = 0, \hat{C}_k\hat{C}_k^\top  = I_{n-\hat{n}}$. In this situation, we have that
$$\mathrm{T}_{z_k}\mathcal{M} = \{\hat{C}_k^\top s: s\in\mathbb{R}^{n-\hat{n}}\}.$$
With this notation, we can rewrite \eqref{eq_imanpl} equivalently as an unconstrained problem: 
\begin{equation}\label{primal-problem}
    \min_s \ H_k(s) = f(z_k) + (\hat{C}_k c_k)^\top s + h\left((B_k\hat{C}_k^\top )s + d_k\right) + \frac{1}{2t_k}\|s\|_2^2,
\end{equation}
whose dual problem is:
\begin{equation}\label{dual-problem}
\max_\lambda \ D_k(\lambda) = f(z_k)-\frac{t_k}{2}\|\hat{C}_k(B_k^\top \lambda + c_k)\|_2^2 - h_\star(\lambda) + \lambda^\top d_k,\lambda\in\mathbb{R}^m,
\end{equation}
where $h_\star: \mathbb{R}^m\rightarrow \mathbb{R}\cup\{+\infty\}$ is the convex conjugate of $h$: $h_\star(\lambda) = \sup_{\lambda'\in\mathbb{R}^m} \lambda'^\top \lambda - h(\lambda').$ The link functions of \eqref{primal-problem} and \eqref{dual-problem} are 
\[s_k(\lambda) = -t_k\left(\hat{C}_k(B_k^\top \lambda + c_k)\right), \quad \mbox{and} \quad x_k(\lambda) = -t_k\left(\hat{C}_k^\top \hat{C}_k(B_k^\top \lambda + c_k)\right)\in \mathrm{T}_{z_k}\mathcal{M},
\]
and $H_k(s_k(\lambda)) = F_{t_k}(x_k(\lambda) + z_k;z_k).$
Noticing that $\hat{C}_k^\top \hat{C}_k = I_n - C_k^\top C_k$, we also have that
$$D_k(\lambda) = \tcb{f(z_k)}-\frac{t_k}{2}(B_k^\top \lambda + c_k)^\top (I_n - C_k^\top C_k)(B_k^\top \lambda + c_k) - h_\star(\lambda) + \lambda^\top d_k,$$ $$x_k(\lambda) = -t_k\left((I_n - C_k^\top C_k)(B_k^\top \lambda + c_k)\right).$$
This indicates that we do not need to explicitly calculate $\hat{C}_k$ in practice. Once we have an inexact dual solution $\lambda_k$ for \eqref{dual-problem}, we can get an inexact primal solution 
$$\tilde{z}_{k+1} = x_k(\lambda_k) + z_k\in \mathrm{T}_{z_k}\mathcal{M} + z_k.$$
By weak duality, we have the following verifiable stopping conditions that imply \eqref{low-high0}-(LACC) and \eqref{low-high0}-(HACC) respectively.
\begin{subequations}
    \begin{align}\label{low1}
	F_{t_k}(x_k(\lambda_k)+z_k;z_k) - D_k(\lambda_k) & \leq \rho_l\left(F(z_k) - F_{t_k}(x_k(\lambda_k)+z_k;z_k)\right), \ \rho_l>0, \\
    \label{high1}
	F_{t_k}(x_k(\lambda_k)+z_k;z_k) - D_k(\lambda_k) & \leq \frac{\rho_h}{2t_k}\|x_k(\lambda_k)\|_2^2, \ \rho_h\in (0,1/4).
\end{align}
\end{subequations}

\subsection{Oracle Complexity}

Throughout this subsection, we assume the nondegeneracy condition
\[
\|B_k(I_n-C_k^\top C_k)B_k^\top\|_2>0
\]
for every outer iteration $k$ at which Algorithm \ref{Nesterov_erg} is called. This ensures that the scalar step size in \eqref{eq_sub_stepsize_chosen} is well defined; it does not require the matrix to be invertible or its norm to have a uniform positive lower bound. The outer convergence results in Section \ref{sec:converge_main} do not require this additional condition.

Let $\dom h_\star:= \{\lambda\in\mathbb{R}^m: h_\star(\lambda)<\infty\}$. Based on discussions in the previous subsection, the complete description of the accelerated proximal gradient method for solving \eqref{dual-problem} is given in Algorithm \ref{Nesterov_erg}. The ergodic iterate $\lambda_{c,erg}$ is returned as the approximate dual solution to find the primal solution $x_k(\lambda_{c,erg})$.
\begin{algorithm}[ht]
	\caption{Accelerated Proximal Gradient (APG) for Solving \eqref{dual-problem}}
	\label{Nesterov_erg}
	\begin{algorithmic}
		\STATE {\bf Input:} $z_k\in\M$, $t_k>0$,  $\lambda^0\in\dom h_\star$, $\lambda_{c,erg}^{0} = \lambda_a^0=\lambda_b^0=\lambda_c^0 = \lambda^0$, $\gamma_0=1$, $\rho_l>0$ and $\rho_h\in (0,1/4)$, inexact type (IT) = (LACC) or (HACC).
		\FOR{$j = 0,1,2\ldots$}
		\STATE 
		\begin{align*}
		\lambda_c^{j+1} & = (1-\gamma_j)\lambda_a^j+\gamma_j\lambda_b^j, \\
		g^{j+1} & = t_kB_{k}(I_n - C_k^\top C_k)(B_k^\top\lambda_c^{j+1} + c_k) - d_k,\\
		\lambda_b^{j+1} &  = \argmin_{\lambda\in\mathbb{R}^m} \ \frac{\gamma_j}{2t_{kj}}\|\lambda - \lambda_b^{j}\|^2_2 + (\lambda - \lambda_c^{j+1})^\top g^{j+1} + h_\star(\lambda), \ t_{kj}>0, \\
		\lambda_a^{j+1} &  = (1-\gamma_j)\lambda_a^j+\gamma_j\lambda_b^{j+1},\\
		\gamma_{j+1} & =2/\left(1+\sqrt{1+4/\gamma_j^2}\right).
		\end{align*}
		\STATE Terminate if one of the following stopping criteria is satisfied (use \eqref{low2_erg} if (IT) = (LACC) and use \eqref{high2_erg} if (IT) = (HACC)). Here, $\lambda_{c,erg}^{j+1} = \left(\sum_{j'=0}^{j} \gamma_{j'}^{-1}\lambda_{c}^{j'+1}\right)/\left(\sum_{j'=0}^{j} \gamma_{j'}^{-1}\right)$. 
		\begin{align}\label{low2_erg}
			F_{t_k}(x_k(\lambda_{c,erg}^{j+1})+z_k;z_k) - D_k(\lambda_a^{j+1}) & \leq \rho_l(F(z_k) - F_{t_k}(x_k(\lambda_{c,erg}^{j+1})+z_k;z_k)) \\
		\label{high2_erg}
			F_{t_k}(x_k(\lambda_{c,erg}^{j+1})+z_k;z_k) - D_k(\lambda_a^{j+1}) & \leq \frac{\rho_h}{2t_k}\|x_k(\lambda_{c,erg}^{j+1})\|_2^2.
		\end{align}
		\ENDFOR
		\STATE {\bf Output:} $x_k^j=x_k(\lambda_{c,erg}^{j+1})$,  $\lambda_k^j = \lambda_a^{j+1}, \Tilde{z}_{k+1} =  x_k^j+ z_k$.
	\end{algorithmic}
\end{algorithm}
The step size $t_{kj}$ in Algorithm \ref{Nesterov_erg} can be chosen as
\begin{equation}\label{eq_sub_stepsize_chosen}
	t_{kj} = (t_k\|B_k(I_n - C_k^\top C_k)B_k^\top \|_2)^{-1},\forall j\in\mathbb{N}.
\end{equation}

The following lemma follows from the primal-dual argument in Corollary 1(b) of \cite{tseng2008accelerated}, applied to the minimization of $-D_k$. 
\begin{lemma}
	\label{rate_subproblem_updated}
	Denote $D_{h_\star} = \sup_{\lambda_1, \lambda_2\in \dom h_\star} \|\lambda_1-\lambda_2\|_2.$
	For a fixed $k$ satisfying the above nondegeneracy condition, Algorithm \ref{Nesterov_erg} with $t_{kj}$ chosen in \eqref{eq_sub_stepsize_chosen} admits a numerical constant $C>0$, independent of $k$ and $j$, such that 
	$x_k^j\in \mathrm{T}_{z_k}\mathcal{M},\forall j\in\mathbb{N},$ and 
	\begin{equation}\label{ineq_sub_rate}
		F_{t_k}(x_k^j+z_k;z_k) - D_k(\lambda_{k}^j)\leq \frac{Ct_k\|B_{k}(I_n - C_k^\top C_k)B_k^\top\|_2D_{h_\star}^2}{(j+1)^2},\forall j\in\mathbb{N}_+.
	\end{equation}
\end{lemma}
\begin{proof}
Write $-D_k(\lambda)=g_k(\lambda)+h_\star(\lambda)$, where the smooth convex quadratic $g_k$ has gradient Lipschitz constant $L_k=t_k\|B_k(I_n-C_k^\top C_k)B_k^\top\|_2$. The saddle representation is
\[
g_k(\lambda)=\max_s\left\{-f(z_k)-(\hat{C}_kc_k)^\top s-\frac{\|s\|_2^2}{2t_k}-\lambda^\top(B_k\hat{C}_k^\top s+d_k)\right\}.
\]
The maximizer is $s_k(\lambda)$, and its weighted average is $s_k(\lambda_{c,erg}^{j+1})$ because $s_k$ is affine. The proof of Corollary 1(b) in \cite{tseng2008accelerated} therefore bounds the primal-dual gap by $\gamma_j^2 L_k\sup_{\lambda\in\dom h_\star}\|\lambda-\lambda^0\|_2^2/2$. Since $\gamma_j\leq 2/(j+2)$ and $\lambda^0\in\dom h_\star$, this implies \eqref{ineq_sub_rate} with, for example, $C=2$. This argument is pointwise on $\dom h_\star$ and takes a supremum instead of assuming that the maximum is attained; thus closedness of the bounded domain is unnecessary. Each proximal update exists because $h_\star$ is proper and lower semicontinuous and the proximal quadratic is strongly convex. Finally, $x_k^j\in\mathrm{T}_{z_k}\mathcal{M}$ follows from its link function.
\end{proof}
\begin{remark}
    Under Assumption \ref{ass:problem}(b), which requires that $h(\cdot)$ is convex and $L_h-$Lipschitz continuous, we know that $\dom h_\star\subseteq \{\lambda\in\mathbb{R}^m:\|\lambda\|_2\leq L_h\}$. Thus, $D_{h_\star}\leq 2L_h$. Also note that each iteration of Algorithm \ref{Nesterov_erg} requires computing a proximal mapping of $h_\star$, and we call it a first-order oracle. 
\end{remark}

We now discuss the overall first-order oracle complexity of Algorithm \ref{alg:IManPL-ls} with subproblems \eqref{eq_imanpl} solved by Algorithm \ref{Nesterov_erg}. 
We use $J_{k}$ to denote the number of first-order oracle calls in the $k$-th iteration of calling Algorithm \ref{Nesterov_erg}. Let $k_\epsilon:=\min\{k\geq 0:\|z_k-S_t(z_k)\|_2/t\leq\epsilon\}$ be the index of the first $\epsilon-$stationary iterate of Algorithm \ref{alg:IManPL-ls} with fixed step size $t>0$. It is the number of outer updates needed to generate that iterate. We count the first-order oracle calls used for those updates, excluding any additional work to certify stationarity at the terminal iterate. Thus,
	$J(\epsilon) := \sum_{k=0}^{k_\epsilon - 1}J_{k}$. 
We are now ready to present the main result of first-order oracle complexity. 
\begin{theorem}\label{thm_overall_complexity0}
	For Algorithm \ref{alg:IManPL-ls} with $t_k = t>0$ and the subproblems \eqref{eq_imanpl} solved by Algorithm \ref{Nesterov_erg} with \eqref{eq_sub_stepsize_chosen}, assume the above nondegeneracy condition at every outer iteration where APG is called. Then the following conclusions hold for any $\epsilon>0$.
	
	(a) When (IT) = (LACC), we have that
	$$J(\epsilon)\leq \left\lfloor\frac{F_0 - F_\star}{t^2c_{l}\epsilon^2}\right\rfloor\left\lceil j_{\epsilon,l}\right\rceil, \mbox{ where } j_{\epsilon,l} = \max\left\{2,\frac{1}{\epsilon}\sqrt{\frac{(1+\rho_l)\tilde{C}}{\rho_l}}\right\}.$$
	
	(b) When (IT) = (HACC), we have that
	$$J(\epsilon)\leq \left\lfloor \frac{F_0 - F_\star}{t^2c_{h}\epsilon^2}\right\rfloor\left\lceil j_{\epsilon,h}\right\rceil, \mbox{ where } j_{\epsilon,h} = \max\left\{2,\frac{1}{\epsilon}\sqrt{\frac{(1+\sqrt{\rho_h})^2\tilde{C}}{\rho_h}}\right\}.$$
	Here, 
    $\tilde{C} := 2C(\sup_{z\in\mathcal{M}} \|\nabla c(z)\|_2)^2D_{h_\star}^2$.
\end{theorem}
\begin{proof}[Proof of Theorem \ref{thm_overall_complexity0}]	
	(a) 
	Denote 
	$$j' = \max\left\{0,\left\lceil\frac{1}{\epsilon}\sqrt{\frac{(1+\rho_l)\tilde{C}}{\rho_l}}\right\rceil - 2\right\}$$
	and we consider any $k<k_\epsilon$. By Lemma \ref{rate_subproblem_updated}, we have 
	$$F_t(x_k^{j'+1}+z_k;z_k) - D_k(\lambda_k^{j'+1})\leq \frac{Ct\|B_{k}(I_n - C_k^\top C_k)B_k^\top\|_2D_{h_\star}^2}{(j'+2)^2}.$$
	Noticing that $(j'+2)^2\geq \frac{(1+\rho_l)\tilde{C}}{\epsilon^2\rho_l}$ and $(\sup_{z\in\mathcal{M}} \|\nabla c(z)\|_2)^2\geq \|B_k\|_2^2\geq \|B_{k}(I_n - C_k^\top C_k)B_k^\top\|_2$, we have
	\begin{equation}\label{proof_upper_gap_a}
		F_t(x_k^{j'+1}+z_k;z_k) - D_k(\lambda_k^{j'+1})\leq \frac{t\rho_l\epsilon^2}{2(1+\rho_l)}.
	\end{equation}
    Note that in the $k$-th iteration, we have not found an $\epsilon$-stationary point yet. Therefore, from Lemma \ref{strong_convex_Ft} we have
	\[
    F_t(z_k;z_k) - \min_{z\in \mathrm{T}_{z_k}\mathcal{M}+z_k} F_t(z;z_k)\geq \frac{t\epsilon^2}{2},
    \]
    which, together with \eqref{proof_upper_gap_a}, yields
	\[F_t(x_k^{j'+1}+z_k;z_k) - D_k(\lambda_k^{j'+1})\leq \frac{\rho_l}{1+\rho_l}\left(F_t(z_k;z_k) - \min_{z\in \mathrm{T}_{z_k}\mathcal{M}+z_k} F_t(z;z_k)\right).
    \]
    This means that 
    \begin{align*}
        & F_t(x_k^{j'+1}+z_k;z_k) - D_k(\lambda_k^{j'+1}) \\
	\leq & \rho_l(-F_t(x_k^{j'+1}+z_k;z_k) + D_k(\lambda_k^{j'+1}) + F_t(z_k;z_k) - \min_{z\in \mathrm{T}_{z_k}\mathcal{M}+z_k} F_t(z;z_k)).
    \end{align*}
    Weak duality of \eqref{eq_imanpl} and \eqref{dual-problem} 
    \[
    D_k(\lambda_k^{j'+1})  - \min_{z\in \mathrm{T}_{z_k}\mathcal{M}+z_k} F_t(z;z_k)\leq 0,
    \]
    Therefore, we further have
	\[
    F_t(x_k^{j'+1}+z_k;z_k) - D_k(\lambda_k^{j'+1})\leq  \rho_l\left(F_t(z_k;z_k) - F_t(x_k^{j'+1}+z_k;z_k)\right).
    \]
    The iterate $x_k^{j'+1}$ is produced after $j'+2$ proximal oracle calls because Algorithm \ref{Nesterov_erg} starts at $j=0$. Thus, by \eqref{low2_erg}, 
	$J_k\leq j'+2=\left\lceil j_{\epsilon,l}\right\rceil,\forall k<k_\epsilon.$ Together with Theorem \ref{thm_main_iteration_complexity} for (IT) = (LACC) that bounds $k_\epsilon$ and $J(\epsilon) := \sum_{k=0}^{k_\epsilon - 1}J_{k}$, we obtain the desired conclusion in part (a).
	
	(b) Denote 
	$$j'' = \max\left\{0,\left\lceil\frac{1}{\epsilon}\sqrt{\frac{(1+\sqrt{\rho_h})^2\tilde{C}}{\rho_h}}\right\rceil - 2\right\}$$
	and we again consider any $k<k_\epsilon$.
	Similar to finding \eqref{proof_upper_gap_a}, we have 
	\begin{equation}\label{proof_upper_gap_b}
		F_t(x_k^{j''+1}+z_k;z_k) - D_k(\lambda_k^{j''+1})\leq \rho_h\|z_k-S_{t}(z_k)\|_2^2/\left(2t(1+\sqrt{\rho_h})^2\right).
	\end{equation}
	By Lemma \ref{strong_convex_Ft} and weak duality, we have
	\[
    \frac{1}{2t}\|x_k^{j''+1}+z_k - S_t(z_k)\|_2^2\leq F_t(x_k^{j''+1}+z_k;z_k) - D_k(\lambda_k^{j''+1}),
    \]
    which implies  
	\begin{equation}\label{eq1_overall}
		\|x_k^{j''+1}+z_k - S_t(z_k)\|_2^2\leq \rho_h\|z_k-S_{t}(z_k)\|_2^2/(1+\sqrt{\rho_h})^2.
	\end{equation}
	By the Cauchy-Schwarz inequality, we have 
	\begin{equation}\label{eq_cauchy1}
		\frac{\rho_h}{2t}\|x_k^{j''+1}\|_2^2 
		\geq \frac{\rho_h}{2(1+\sqrt{\rho_h})t}\|z_k 
		- S_t(z_k)\|_2^2- \frac{\sqrt{\rho_h}}{2t}\|z_k + x_k^{j''+1} - S_t(z_k)\|_2^2.
	\end{equation}
	By \eqref{eq1_overall}, we have
	\[
    \frac{\rho_h}{2(1+\sqrt{\rho_h})t}\|z_k 
	- S_t(z_k)\|_2^2- \frac{\sqrt{\rho_h}}{2t}\|z_k + x_k^{j''+1} - S_t(z_k)\|_2^2\geq \rho_h\|z_k - S_t(z_k)\|_2^2/\left(2t(1+\sqrt{\rho_h})^2\right).
    \]
    Thus, together with \eqref{proof_upper_gap_b}
	\begin{equation*}
		F_t(x_k^{j''+1}+z_k;z_k) - D_k(\lambda_k^{j''+1})\leq \rho_h\|z_k-S_{t}(z_k)\|_2^2/\left(2t(1+\sqrt{\rho_h})^2\right)\leq \frac{\rho_h}{2t}\|x_k^{j''+1}\|_2^2.
	\end{equation*}
	Thus, by \eqref{high2_erg}, 
	$$J_k\leq j''+2=\left\lceil j_{\epsilon,h}\right\rceil,\forall k<k_\epsilon.$$ Together with Theorem \ref{thm_main_iteration_complexity} for (IT) = (HACC) that bounds $k_\epsilon$ and $J(\epsilon) := \sum_{k=0}^{k_\epsilon - 1}J_{k}$, we obtain the desired conclusion in part (b).
\end{proof}

\tcb{Theorems \ref{thm_overall_complexity0} and \ref{thm_main_iteration_complexity} show that 
    both (LACC) and (HACC) lead to a main iteration complexity of $O(1/\epsilon^2)$ and an overall iteration complexity of $O(1/\epsilon^3)$. This highlights the similarity of the conditions for IManPL. Generally, we should prefer (HACC) under scenarios where obtaining a high-accuracy solution for the subproblems is easy and solving the subproblems to a high accuracy is beneficial. In contrast, (LACC) should be preferred when solving computationally heavy subproblems, and when a high-accuracy solution would not yield significant gains in the main iterations.}

\section{Numerical Experiments}\label{sec:numerical}
In this section, we apply our IManPL algorithm (Algorithm \ref{alg:IManPL-ls} with subproblem \eqref{eq_imanpl} solved by Algorithm \ref{Nesterov_erg}) to solve the sparse spectral clustering (SSC) problem \eqref{eq_ssc} and compare IManPL with ManPL \citep{wang2022manifold} that focuses on the same objective function \eqref{eq_ssc}. Due to the space constraint, we present the numerical results on solving the sparse PCA problem in the appendix. 
Following \cite{wang2022manifold}, we use the same implementation of ManPL and use the same synthetic and real data explored in \cite{park2018spectral} and \cite{wang2022manifold}.
Throughout this section, we assume that $r$ (the number of classes) in \eqref{eq_ssc} is known. All the codes were written in MATLAB and executed on a server with an Intel(R) Xeon(R) Gold 6226R CPU at 2.90 GHz. Each task is limited to 10 cores and 128 GB of memory. For our IManPL, we show the performance of both (IT) = (LACC) and (IT) = (HACC). We set $\rho_l = \rho_h = 0.2$. The complexity results above concern fixed outer step sizes; the adaptive choices used here and in Appendix \ref{sec:numerical_spca} are empirical variants. Motivated by \cite{chen2020proximal}, we determine the step size of IManPL in an adaptive manner:
\begin{equation}\label{eq_adaptive_stepsize}
	t_{k+1} = \begin{cases}
		\alpha_kt_k, & \alpha_k<1,\\
		2t_k, & \alpha_k = 1.
	\end{cases}
\end{equation}
We will compare the CPU time of our IManPL with that of ManPL. Our numerical experiment consistently shows that IManPL is more efficient than ManPL, and using (LACC) and (HACC) for IManPL performs similarly.

The remaining content of Section \ref{sec:numerical} is summarized as follows. Section \ref{subsec:implementation_ssc} provides the implementation details of IManPL for SSC and demonstrates its lower per-iteration computational complexity relative to ManPL.
Section \ref{comp_manpl} compares IManPL and ManPL. 

\subsection{Implementation Details of IManPL for SSC}\label{subsec:implementation_ssc}
In \eqref{eq_ssc}, $\mathcal{M} = \mbox{St}(N,r)$ ($N\geq r$). For $U\in\mathcal{M}$, $\mathrm{T}_{U}\mathcal{M} = \{V\in\mathbb{R}^{N\times r}: V^\top U + U^\top V = 0_{r\times r}\}$. Treating $\mbox{vec}(U)$ as the input for \eqref{problem_single} so that $n = Nr$ and $m = N^2$, we further have $f(\mbox{vec}(U)) = \left\langle U, SU\right\rangle, h(\cdot) = \kappa\|\cdot\|_1$, $c(\mbox{vec}(U)) = \mbox{vec}(UU^\top)$, $\nabla f(\mbox{vec}(U)) = 2\mbox{vec}(SU)$, and  $\nabla c(\mbox{vec}(U)) = (I_{N^2} + K_{NN})(U\otimes I_N)$ where $K_{NN}$ is the commutation matrix for $N\times N$ matrices. Next, we verify Assumptions \ref{ass:problem}. For (a), $L_f = 2\|S\|_2$ since $2\|S(U_1 - U_2)\|_F\leq 2\|S\|_2\|(U_1 - U_2)\|_F,\forall\  U_1,U_2\in\mathbb{R}^{N\times r}$. For (b), $L_h = \sqrt{N^2}\kappa = N\kappa$. For (c), since $\|I_{N^2} + K_{NN}\|_2\leq 2$ and $\|(U_1 - U_2)\otimes I_N\|_2 = \|U_1 - U_2\|_2\leq \|U_1 - U_2\|_F, \forall \ U_1,U_2\in\mathbb{R}^{N\times r}$, we have $L_c = 2$. For (d), the compactness is straightforward. \cite{chen2020proximal} also provide four examples of the retraction operation: the exponential mapping, the polar decomposition, the QR factorization, and the Cayley transformation. Following \cite{wang2022manifold}, we will always use the QR factorization for numerical experiments on \eqref{eq_ssc}.

For the nontrivial SSC case $1\leq r<N$, the nondegeneracy condition of Section \ref{sec:subsolver_overall} holds at every $U_k$: a nonzero tangent direction $V$ with $U_k^\top V=0$ has $\nabla c(\operatorname{vec}(U_k))\operatorname{vec}(V)=\operatorname{vec}(U_kV^\top+VU_k^\top)\neq0$. Hence $B_k$ does not vanish on the tangent space and $\|B_k(I_n-C_k^\top C_k)B_k^\top\|_2>0$. When $r=N$, $UU^\top=I_N$ and the SSC objective is constant on $\St(N,N)$, so no iterative subproblem solve is needed.

Next, we show some details when implementing Algorithm \ref{alg:IManPL-ls} with the subproblem \eqref{eq_imanpl} solved by Algorithm \ref{Nesterov_erg}. Denote $U_k\in\mathcal{M}$ as the estimated solution to \eqref{eq_ssc} at the start of the $k-$th outer iteration. With some mild abuse of notation, \eqref{eq_imanpl} corresponds to
\begin{equation}\label{eq_ssc_sub_primal}
	\min_{V\in \mathrm{T}_{U_k}\mathcal{M}} F_{t_k}(U_k + V;U_k) = \mbox{tr}(U_k^\top SU_k) + \langle V, 2SU_k\rangle + \frac{1}{2t_k}\|V\|_F^2 + \kappa\|\mbox{vec}(U_kU_k^\top + U_kV^\top + VU_k^\top)\|_1. 
\end{equation}
Following the discussions in Section \ref{subsec:suff_stopping}, the dual problem is given by
\begin{equation}\label{eq_ssc_sub_dual}
	\min_{\Lambda\in \mathbb{R}^{N\times N}:\|\mbox{vec}(\Lambda)\|_\infty\leq \kappa} -D_k(\Lambda):= -\mbox{tr}(U_k^\top SU_k) +  2t_k\|(I_N - U_kU_k^\top)(S+\hat{\Lambda})U_k\|_F^2 - \langle U_kU_k^\top, \Lambda \rangle
\end{equation}
and the link function is given by
\begin{equation}\label{link_func_ssc}
	V_{k}(\Lambda) = -2t_k(I_N - U_kU_k^\top)(S+\hat{\Lambda})U_k.
\end{equation}
Here, $\hat{\Lambda} = (\Lambda + \Lambda^\top)/2$. In addition,
$-\nabla D_k(\Lambda) = 2t_k(I_N - U_kU_k^\top)(S+\hat{\Lambda})(U_kU_k^\top) + 2t_k(U_kU_k^\top)(S+\hat{\Lambda})(I_N - U_kU_k^\top) - U_kU_k^\top.$ When applying APG (Algorithm \ref{Nesterov_erg}) to solve \eqref{eq_ssc_sub_dual}, the dominant computational cost for each subproblem iteration is to calculate $-\nabla D_k(\Lambda)$. By splitting the matrix multiplications into multiple steps, the cost is $O(N^2r)$. In contrast, each subproblem iteration of ManPL for \eqref{eq_imanpl} requires solving multiple linear systems with $N^2\times N^2$ matrices when using the proximal point algorithm along
with the adaptive regularized semi-smooth Newton method \citep{xiao2018regularized}, which has a much higher computational cost.

\subsection{Comparisons with ManPL}\label{comp_manpl}
Following \cite{wang2022manifold}, we let $\kappa\in \{10^{-2},10^{-3},10^{-4},10^{-5}\}$. For each given dataset, we first implement ManPL under each $\kappa$. Then we apply $K-$Means on each generated $U\in\mathbb{R}^{N\times r}$ 10 times and calculate the corresponding mean NMI score, which measures the clustering accuracy (see Section D of \cite{park2018spectral}). We compare the CPU time of solving \eqref{eq_ssc} for ManPL and IManPL under the best $\kappa$ in terms of the largest mean NMI score. Here, to ensure a fair comparison, we run ManPL first and IManPL second. We terminate IManPL after it finds a solution with a smaller objective function value than ManPL. For any given dataset, we initialize both methods by solving \eqref{eq_ssc} with $\kappa = 0$, which is a standard step in spectral clustering and can be done by the eigendecomposition of the Laplacian matrix $S$. The CPU time for the initialization is not included in the reported results.

\subsubsection{Experiments on Synthetic Datasets.}\label{subsec:synthetic}

Following \cite{park2018spectral} and \cite{wang2022manifold}, we generate two synthetic datasets with $r = 5$ clusters as follows.

\begin{itemize}
	\item \textbf{Synthetic Data 1 (Circle with Gaussian Noise).}  
	We generate $C=5$ clusters in a latent two-dimensional space. The cluster centers are placed uniformly on a circle of radius $R=1$, i.e.,
	\[
	\mathbf{c}_\ell = \big(\cos(2\pi \ell / C), \, \sin(2\pi \ell / C)\big), 
	\quad \ell = 1, \dots, C.
	\]
	For each cluster $\ell$, we sample $n_c=100$ points around $\mathbf{c}_\ell$ by adding Gaussian perturbations:
	\[
	\mathbf{x}_{\ell,i} = \mathbf{c}_\ell + \sigma \cdot \boldsymbol{\epsilon}_{\ell,i}, 
	\quad \boldsymbol{\epsilon}_{\ell,i}\sim \mathcal{N}(\mathbf{0}, I_2),
	\]
	where $\sigma = 0.3 R$ controls the noise scale. Concatenating all clusters gives $N=C n_c=500$ samples in $\mathbb{R}^2$.  
	To embed the data into a higher-dimensional space, we draw a random Gaussian projection matrix $P\in\mathbb{R}^{p\times 2}$ with $p=10$, and compute $\tilde{X}=XP^\top$. Finally, we add an additional heterogeneous Gaussian noise matrix $\eta\in\mathbb{R}^{N\times p}$ with entries drawn from $0.3R\cdot \mathcal{N}(0,1)$. The final dataset is
	\[
	X = \tilde{X} + \eta \in \mathbb{R}^{500 \times 10}.
	\]
	
	\item \textbf{Synthetic Data 2 (Low-Dimensional Linear Mixture).}  
	We construct a latent basis $B'\in\mathbb{R}^{C\times d}$ with $C=5$ clusters and $d=5$ latent dimensions. Each row of $B'$ is sampled independently from Gaussian distributions with heterogeneous variances (to model cluster imbalance). We extend this to a $p=10$ dimensional ambient space by defining
	\[
	B = [B', \, 0_{C\times (p-d)}] \in \mathbb{R}^{C\times p}.
	\]
	Each data point is assigned a cluster label $z_i \in \{1,\dots,C\}$ drawn uniformly at random. We encode the labels in an indicator matrix $Z\in\mathbb{R}^{N\times C}$ with $Z_{i\ell}=1$ if $z_i=\ell$ and $0$ otherwise.  
	The observed data are then generated as
	\[
	X = ZB + W,
	\]
	where $W\in\mathbb{R}^{N\times p}$ is a noise matrix with i.i.d.\ entries sampled from $\mathcal{N}(0,\sigma^2)$, with $\sigma=0.2R$ to represent $20\%$ noise relative to the embedding radius. We set $N=500$ throughout.
\end{itemize}

\cite{park2018spectral} and \cite{wang2022manifold} suggest constructing multiple similarity matrices to form the SSC problem. Thus, following Section 5 of \cite{wang2022manifold}, we apply the same strategy to generate 55 similarity matrices and construct the Laplacian matrix $S$ in \eqref{eq_ssc} based on the mean of the 55 matrices.
Table \ref{table:single_kernel_syndata_manpl_imanpl} summarizes the comparison and shows the advantage of IManPL. The table also shows that (LACC) and
(HACC) perform similarly.
\begin{table}[htbp]
	\centering
	\begin{tabular}{|c|c|c|c|}
		\hline
		Datasets & ManPL & IManPL-(LACC) & IManPL-(HACC) \\
		\hline
		Synthetic data 1 & 3.32 & 2.20  & \textbf{2.13} \\
        \hline
		Synthetic data 2 & 7.13 & 3.19  & \textbf{2.52} \\
		\hline
	\end{tabular}
	\caption{Comparison of CPU time between ManPL and IManPL for synthetic datasets. The best one for each dataset is highlighted.}
	\label{table:single_kernel_syndata_manpl_imanpl}
\end{table}

\subsubsection{Experiments on Single-cell RNA Sequence (scRNA-seq) Data.}\label{subsubsec:single_cell}
Clustering cells and identifying subgroups are important in high-dimensional scRNA-seq data analysis \cite{park2018spectral}. In what follows, we focus on experiments for candidate algorithms
to cluster scRNA-seq data on nine real datasets used in \cite{park2018spectral} and \cite{wang2022manifold}. Table \ref{scRNA-seq-single_dataset_des} summarizes the datasets.
\begin{table}[htbp]
	\centering
	\begin{tabular}{cccc}
		\toprule 
		Dataset &  Sample Size $(N)$& Dimension $(p)$ & Classes $(r)$\\
		\midrule
        \cite{treutlein2014reconstructing} & 80 & 959 & 5 \\
        \cite{ting2014single} & 114 & 14405 & 5\\
		\cite{deng2014single} & 135 & 12548 & 7\\
		\cite{buettner2015computational} & 182 & 8989 & 3\\
		\cite{pollen2014low} & 249 & 14805 & 11\\
        \cite{schlitzer2015identification} & 251 & 11834 & 3 \\
		\cite{tasic2016adult} & 1727 & 5832 & 49 \\
		\cite{zeisel2015cell} & 3005 & 4412 & 48 \\
		\cite{macosko2015highly} & 6418 & 12822 & 39 \\
		\bottomrule 
	\end{tabular}
	\caption{Description of several scRNA-seq datasets}
	\label{scRNA-seq-single_dataset_des}
\end{table}
We construct the Laplacian matrix $S$ using the same method as in Section \ref{subsec:synthetic}. 
Table \ref{table:single_compare_cputime2_equalweight} summarizes the comparison and shows the advantage of IManPL. The table also shows that (LACC) and
(HACC) perform similarly. 
\begin{table}[H]
	\centering
	\begin{tabular}{|c|c|c|c|}
		\hline
		Datasets & ManPL & IManPL-(LACC) & IManPL-(HACC) \\
		\hline
        \cite{treutlein2014reconstructing}  & 0.021 & \textbf{0.012} & 0.019 \\
        \cite{ting2014single} & 0.59 & 0.23 & \textbf{0.12} \\
        \cite{deng2014single}  & 0.15 & 0.073 & \textbf{0.070} \\
		\cite{buettner2015computational}  &  0.30 & 0.16 & \textbf{0.15} \\
		\cite{pollen2014low}  & 4.98 & 2.13 & \textbf{1.29} \\
		\cite{schlitzer2015identification} & 0.24 & \textbf{0.074} & 0.14 \\
        \cite{tasic2016adult}  & 178.46 & \textbf{8.12} & 12.99 \\
        \cite{zeisel2015cell} & 225.28 & \textbf{11.40} & 12.37 \\
		\cite{macosko2015highly} & 902.55 & 51.61 & \textbf{48.53} \\
		\hline
	\end{tabular}
	\caption{Comparison of CPU time between ManPL and IManPL for scRNA-seq datasets. The smallest record for each dataset is highlighted.}
	\label{table:single_compare_cputime2_equalweight}
\end{table}

\section{Conclusion}\label{sec:conclusion}
We proposed the inexact manifold proximal linear (IManPL) algorithm for nonsmooth Riemannian composite optimization, which introduces adaptive subproblem stopping conditions. For fixed positive outer step sizes, our analysis established the $O(1/\epsilon^2)$ main iteration complexity and stationarity of accumulation points. Under the additional APG nondegeneracy condition, the overall first-order oracle complexity is $O(1/\epsilon^3)$, matching the best-known order of related methods under their respective assumptions and stationarity measures. Numerical experiments on sparse spectral clustering and sparse principal component analysis confirm that IManPL outperforms ManPL or ManPG in computational efficiency. These results highlight the value of adaptive inexactness and broaden the applicability of proximal methods to large-scale manifold optimization problems.
\bibliographystyle{plainnat}
\bibliography{inexact_pl_refs,reference_arv}

@article{drusvyatskiy2019efficiency,
  title={Efficiency of minimizing compositions of convex functions and smooth maps},
  author={Drusvyatskiy, Dmitriy and Paquette, Courtney},
  journal={Mathematical Programming},
  volume={178},
  number={1},
  pages={503--558},
  year={2019},
  publisher={Springer}
}

@article{zheng2023new,
  title={A New Inexact Proximal Linear Algorithm With Adaptive Stopping Criteria for Robust Phase Retrieval},
  author={Zheng, Zhong and Ma, Shiqian and Xue, Lingzhou},
  journal={IEEE Transactions on Signal Processing},
  volume={72},
  pages={1081--1093},
  year={2024},
  publisher={IEEE}
}

@unpublished{tseng2008accelerated,
  title={On accelerated proximal gradient methods for convex-concave optimization},
  author={Tseng, Paul},
  note={Manuscript, University of Washington, May 2008},
  url={https://www.mit.edu/~dimitrib/PTseng/papers/apgm.pdf},
  year={2008}
}

@article{chen2020proximal,
  title={Proximal gradient method for nonsmooth optimization over the {S}tiefel manifold},
  author={Chen, Shixiang and Ma, Shiqian and So, Anthony Man-Cho and Zhang, Tong},
  journal={SIAM Journal on Optimization},
  volume={30},
  number={1},
  pages={210--239},
  year={2020},
  publisher={SIAM}
}

@article{lu2016convex,
  title={Convex sparse spectral clustering: Single-view to multi-view},
  author={Lu, Canyi and Yan, Shuicheng and Lin, Zhouchen},
  journal={IEEE Transactions on Image Processing},
  volume={25},
  number={6},
  pages={2833--2843},
  year={2016},
  publisher={IEEE}
}

@article{lu2018nonconvex, title={Nonconvex Sparse Spectral Clustering by Alternating Direction Method of Multipliers and Its Convergence Analysis}, 
volume={32}, 
number={1}, 
pages={3714-3721},
journal={Proceedings of the AAAI Conference on Artificial Intelligence}, 
author={Lu, Canyi and Feng, Jiashi and Lin, Zhouchen and Yan, Shuicheng}, year={2018}, month={Apr.} }

@article{park2018spectral,
  title={Spectral clustering based on learning similarity matrix},
  author={Park, Seyoung and Zhao, Hongyu},
  journal={Bioinformatics},
  volume={34},
  number={12},
  pages={2069--2076},
  year={2018},
  publisher={Oxford University Press}
}

@article{chen2019alternating,
  title={An alternating manifold proximal gradient method for sparse principal component analysis and sparse canonical correlation analysis},
  author={Chen, Shixiang and Ma, Shiqian and Xue, Lingzhou and Zou, Hui},
  year={2020},
  journal={INFORMS Journal on Optimization},
  volume={2},
  number={3},
  pages={192--208},
  publisher={INFORMS}
}

@article{bolte2014proximal,
  title={Proximal alternating linearized minimization for nonconvex and nonsmooth problems},
  author={Bolte, J{\'e}r{\^o}me and Sabach, Shoham and Teboulle, Marc},
  journal={Mathematical Programming},
  volume={146},
  number={1},
  pages={459--494},
  year={2014},
  publisher={Springer}
}

@article{huang2022riemannian,
  title={Riemannian proximal gradient methods},
  author={Huang, Wen and Wei, Ke},
  journal={Mathematical Programming},
  volume={194},
  number={1},
  pages={371--413},
  year={2022},
  publisher={Springer}
}

@article{huang2023inexact,
  title={An inexact {R}iemannian proximal gradient method},
  author={Huang, Wen and Wei, Ke},
  journal={Computational Optimization and Applications},
  volume={85},
  number={1},
  pages={1--32},
  year={2023},
  publisher={Springer}
}

@article{huang2022riemannian2,
  title={A {R}iemannian optimization approach to clustering problems},
  author={Huang, Wen and Wei, Meng and Gallivan, Kyle A and Van Dooren, Paul},
  journal={Journal of Scientific Computing},
  volume={103},
  number={1},
  pages={8},
  year={2025},
  publisher={Springer}
}

@article{deng2014single,
  title={Single-cell {RNA}-seq reveals dynamic, random monoallelic gene expression in mammalian cells},
  author={Deng, Qiaolin and Ramsk{\"o}ld, Daniel and Reinius, Bj{\"o}rn and Sandberg, Rickard},
  journal={Science},
  volume={343},
  number={6167},
  pages={193--196},
  year={2014},
  publisher={American Association for the Advancement of Science}
}

@article{ting2014single,
  title={Single-cell {RNA} sequencing identifies extracellular matrix gene expression by pancreatic circulating tumor cells},
  author={Ting, David T and Wittner, Ben S and Ligorio, Matteo and Jordan, Nicole Vincent and Shah, Ajay M and Miyamoto, David T and Aceto, Nicola and Bersani, Francesca and Brannigan, Brian W and Xega, Kristina and others},
  journal={Cell Reports},
  volume={8},
  number={6},
  pages={1905--1918},
  year={2014},
  publisher={Elsevier}
}

@article{treutlein2014reconstructing,
  title={Reconstructing lineage hierarchies of the distal lung epithelium using single-cell {RNA}-seq},
  author={Treutlein, Barbara and Brownfield, Doug G and Wu, Angela R and Neff, Norma F and Mantalas, Gary L and Espinoza, F Hernan and Desai, Tushar J and Krasnow, Mark A and Quake, Stephen R},
  journal={Nature},
  volume={509},
  number={7500},
  pages={371--375},
  year={2014},
  publisher={Nature Publishing Group UK London}
}

@article{buettner2015computational,
  title={Computational analysis of cell-to-cell heterogeneity in single-cell {RNA}-sequencing data reveals hidden subpopulations of cells},
  author={Buettner, Florian and Natarajan, Kedar N and Casale, F Paolo and Proserpio, Valentina and Scialdone, Antonio and Theis, Fabian J and Teichmann, Sarah A and Marioni, John C and Stegle, Oliver},
  journal={Nature Biotechnology},
  volume={33},
  number={2},
  pages={155--160},
  year={2015},
  publisher={Nature Publishing Group US New York}
}

@article{schlitzer2015identification,
  title={Identification of c{DC}1-and c{DC}2-committed {DC} progenitors reveals early lineage priming at the common {DC} progenitor stage in the bone marrow},
  author={Schlitzer, Andreas and Sivakamasundari, V and Chen, Jinmiao and Sumatoh, Hermi Rizal Bin and Schreuder, Jaring and Lum, Josephine and Malleret, Benoit and Zhang, Sanqian and Larbi, Anis and Zolezzi, Francesca and others},
  journal={Nature Immunology},
  volume={16},
  number={7},
  pages={718--728},
  year={2015},
  publisher={Nature Publishing Group US New York}
}

@article{pollen2014low,
  title={Low-coverage single-cell m{RNA} sequencing reveals cellular heterogeneity and activated signaling pathways in developing cerebral cortex},
  author={Pollen, Alex A and Nowakowski, Tomasz J and Shuga, Joe and Wang, Xiaohui and Leyrat, Anne A and Lui, Jan H and Li, Nianzhen and Szpankowski, Lukasz and Fowler, Brian and Chen, Peilin and others},
  journal={Nature Biotechnology},
  volume={32},
  number={10},
  pages={1053--1058},
  year={2014},
  publisher={Nature Publishing Group US New York}
}

@article{tasic2016adult,
  title={Adult mouse cortical cell taxonomy revealed by single cell transcriptomics},
  author={Tasic, Bosiljka and Menon, Vilas and Nguyen, Thuc Nghi and Kim, Tae Kyung and Jarsky, Tim and Yao, Zizhen and Levi, Boaz and Gray, Lucas T and Sorensen, Staci A and Dolbeare, Tim and others},
  journal={Nature Neuroscience},
  volume={19},
  number={2},
  pages={335--346},
  year={2016},
  publisher={Nature Publishing Group}
}

@article{macosko2015highly,
  title={Highly parallel genome-wide expression profiling of individual cells using nanoliter droplets},
  author={Macosko, Evan Z and Basu, Anindita and Satija, Rahul and Nemesh, James and Shekhar, Karthik and Goldman, Melissa and Tirosh, Itay and Bialas, Allison R and Kamitaki, Nolan and Martersteck, Emily M and others},
  journal={Cell},
  volume={161},
  number={5},
  pages={1202--1214},
  year={2015},
  publisher={Elsevier}
}

@article{zeisel2015cell,
  title={Cell types in the mouse cortex and hippocampus revealed by single-cell {RNA}-seq},
  author={Zeisel, Amit and Mu{\~n}oz-Manchado, Ana B and Codeluppi, Simone and L{\"o}nnerberg, Peter and La Manno, Gioele and Jur{\'e}us, Anna and Marques, Sueli and Munguba, Hermany and He, Liqun and Betsholtz, Christer and others},
  journal={Science},
  volume={347},
  number={6226},
  pages={1138--1142},
  year={2015},
  publisher={American Association for the Advancement of Science}
}

@article{pearson1901liii,
  title={LIII. On lines and planes of closest fit to systems of points in space},
  author={Pearson, Karl},
  journal={The London, Edinburgh, and Dublin Philosophical Magazine and Journal of Science},
  volume={2},
  number={11},
  pages={559--572},
  year={1901},
  publisher={Taylor \& Francis}
}

@article{jolliffe2003modified,
  title={A modified principal component technique based on the {LASSO}},
  author={Jolliffe, Ian T and Trendafilov, Nickolay T and Uddin, Mudassir},
  journal={Journal of Computational and Graphical Statistics},
  volume={12},
  number={3},
  pages={531--547},
  year={2003},
  publisher={Taylor \& Francis}
}

@book{chung1997spectral,
  title={Spectral Graph Theory},
  author={Chung, Fan RK},
  volume={92},
  year={1997},
  publisher={American Mathematical Soc.}
}

@article{yang2025theoretical,
  title={Theoretical Guarantees for Sparse Principal Component Analysis based on the Elastic Net},
  author={Yang, Haoyi and Zhang, Teng and Xue, Lingzhou},
  journal={IEEE Transactions on Information Theory},
  year={2025},
  pages={7149 - 7175},
  volume={71},
  publisher={IEEE}
}

@article{xu2024oracle,
  title={On the oracle complexity of a {R}iemannian inexact augmented {L}agrangian method for nonsmooth composite problems over {R}iemannian submanifolds},
  author={Xu, Meng and Jiang, Bo and Liu, Ya-Feng and So, Anthony Man-Cho},
  journal={Optimization Letters},
  year={2025},
  doi={10.1007/s11590-025-02210-8},
  publisher={Springer}
}

@article{zheng2024adaptive,
  title={Adaptive Algorithms for Robust Phase Retrieval},
  author={Zheng, Zhong and Aybat, Necdet Serhat and Ma, Shiqian and Xue, Lingzhou},
  journal={arXiv preprint arXiv:2409.19162},
  year={2024}
}

@article{zhou2022robust,
  title={On the robust isolated calmness of a class of nonsmooth optimizations on {R}iemannian manifolds and its applications},
  author={Bao, Chenglong and Ding, Chao and Zhou, Yuexin},
  journal={Computational Optimization and Applications},
  volume={92},
  pages={709--754},
  doi={10.1007/s10589-025-00712-w},
  year={2025},
  publisher={Springer}
}

@article{xiao2018regularized,
  title={A regularized semi-smooth {N}ewton method with projection steps for composite convex programs},
  author={Xiao, Xiantao and Li, Yongfeng and Wen, Zaiwen and Zhang, Liwei},
  journal={Journal of Scientific Computing},
  volume={76},
  number={1},
  doi={10.1007/s10915-017-0624-3},
  pages={364--389},
  year={2018},
  publisher={Springer}
}

@article{boumal2019global,
  title={Global rates of convergence for nonconvex optimization on manifolds},
  author={Boumal, Nicolas and Absil, Pierre-Antoine and Cartis, Coralia},
  journal={IMA Journal of Numerical Analysis},
  volume={39},
  number={1},
  pages={1--33},
  year={2019},
  publisher={Oxford University Press}
}

@article{li2023riemannian,
  title={A {R}iemannian alternating direction method of multipliers},
  author={Li, Jiaxiang and Ma, Shiqian and Srivastava, Tejes},
  journal={Mathematics of Operations Research},
  volume={50},
  number={4},
  pages={3222--3242},
  year={2025},
  doi={10.1287/moor.2023.0068},
  publisher={INFORMS}
}

@article{beck2023dynamic,
	author = {Beck, A. and Rosset, I.},
	title = {A Dynamic Smoothing Technique for a Class of Nonsmooth Optimization Problems on Manifolds},
	journal = {SIAM J.  Optim.},
	volume = {33},
	number = {3},
	pages = {1473-1493},
	year = {2023}

}

@article{xu2024riemannian,
	author={M. Xu and B. Jiang and Y.-F. Liu and A. M.-C. So},
	title={A {R}iemannian Alternating Descent Ascent
Algorithmic Framework for Nonconvex-Linear
Minimax Problems on {R}iemannian Manifolds}, 
  journal={arXiv preprint arXiv:2409.19588},
	year={2024}
}

@INPROCEEDINGS{xu2023efficient2,
	author={M. Xu and B. Jiang and W. Pu and Y.-F. Liu and A. M.-C. So},
	booktitle={Proc. IEEE Int. Conf. Acoust., Speech, Signal Process.}, 
	title={An Efficient Alternating {R}iemannian/Projected Gradient Descent Ascent Algorithm for Fair Principal Component Analysis}, 
	year={2024},
	volume={},
	number={},
	pages={7195-7199}
}

@article{li2021weakly,
	title={Weakly convex optimization over {S}tiefel manifold using {R}iemannian subgradient-type methods},
	author={X. Li and S. Chen and Z. Deng and Q. Qu and Z. Zhu and A. M.-C. So},
	journal={SIAM J. Optim.},
	volume={31},
	number={3},
	pages={1605--1634},
	year={2021},
	publisher={SIAM}
}

@article{wang2022manifold,
	title={A manifold proximal linear method for sparse spectral clustering with application to single-cell {RNA} sequencing data analysis},
	author={Z. Wang and B. Liu and S. Chen and S. Ma and L. Xue and H. Zhao},
	journal={INFORMS J. Optim.},
	volume={4},
	number={2},
	pages={200--214},
	year={2022},
	publisher={INFORMS}
}

@article{zhou2023semismooth,
	title={A semismooth {N}ewton based augmented {L}agrangian method for nonsmooth optimization on matrix manifolds},
	author={Y. Zhou and C. Bao and C. Ding and J. Zhu},
	journal={Math. Program.},
	volume={201},
	number={1},
	pages={1--61},
	year={2023},
	publisher={Springer}
}

@article{peng2023riemannian,
	title={Riemannian Smoothing Gradient Type Algorithms for Nonsmooth Optimization Problem on Compact {R}iemannian Submanifold Embedded in {E}uclidean Space},
	author={Z. Peng and W. Wu and J. Hu and K. Deng},
	journal={Appl. Math. Optim.},
	volume={88},
	number={3},
	pages={85},
	year={2023},
	publisher={Springer}
}

@inproceedings{samadi2018price,
	author = {S. Samadi and U. Tantipongpipat and J. H. Morgenstern and M. Singh and S. Vempala},
	booktitle = {Proc. Adv. Neural Inf. Process. Syst.},
	pages = {10999--11010},
	title = {The Price of Fair {PCA}: One Extra dimension},
	volume = {31},
	year = {2018}
}

@article{borckmans2014riemannian,
	title={A {R}iemannian subgradient algorithm for economic dispatch with valve-point effect},
	author={P. B. Borckmans and S. E. Selvan and N. Boumal and P.-A. Absil},
	journal={J. Comput. Appl. Math.},
	volume={255},
	pages={848--866},
	year={2014},
	publisher={Elsevier}
}

@article{hosseini2017riemannian,
	title={A {R}iemannian gradient sampling algorithm for nonsmooth optimization on manifolds},
	author={S. Hosseini and A. Uschmajew},
	journal={SIAM J. Optim.},
	volume={27},
	number={1},
	pages={173--189},
	year={2017},
	publisher={SIAM}
}

@article{hosseini2018line,
	title={Line search algorithms for locally {L}ipschitz functions on {R}iemannian manifolds},
	author={S. Hosseini and W. Huang and R. Yousefpour},
	journal={SIAM J. Optim.},
	volume={28},
	number={1},
	pages={596--619},
	year={2018},
	publisher={SIAM}
}

@article{lai2014splitting,
	title={A splitting method for orthogonality constrained problems},
	author={R. Lai and S. Osher},
	journal={J. Sci. Comput.},
	volume={58},
	pages={431--449},
	year={2014},
	publisher={Springer}
}

@inproceedings{kovnatsky2016madmm,
	title={{MADMM}: a generic algorithm for non-smooth optimization on manifolds},
	author={A. Kovnatsky and K. Glashoff and M. M. Bronstein},
	booktitle={Proc. Comput. Vis. ECCV},
	pages={680--696},
	year={2016},
	organization={Springer}
}

@article{chen2024nonsmooth,
	title={Nonsmooth Optimization over the {S}tiefel Manifold and Beyond: Proximal Gradient Method and Recent Variants},
	author={S. Chen and S. Ma and A. M.-C. So and T. Zhang},
	journal={SIAM Rev.},
	volume={66},
	number={2},
	pages={319--352},
	year={2024},
	publisher={SIAM}
}

@book{absil2008optimization,
	title={Optimization Algorithms on Matrix Manifolds},
	author={P.-A. Absil and R. Mahony and R. Sepulchre},
	year={2008},
	publisher={Princeton Univ. Press}
}

@book{boumal2023introduction,
	title={An Introduction to Optimization on Smooth Manifolds},
	author={N. Boumal},
	year={2023},
	publisher={Cambridge Univ. Press}
}

@article{deng2024oracle,
	title={Oracle Complexities of Augmented {L}agrangian Methods for Nonsmooth Composite Optimization on a Compact Submanifold},
	author={K. Deng and J. Hu and J. Wu and Z. Wen},
	journal={Mathematics of Operations Research},
	year={2025},
	doi={10.1287/moor.2024.0498},
	note={Articles in Advance}
}

@article{deng2023manifold,
	title={A Manifold Inexact Augmented {L}agrangian Method for Nonsmooth Optimization on {R}iemannian Submanifolds in {E}uclidean Space},
	author={K. Deng and Z. Peng},
	journal={IMA J. Numer. Anal.},
	volume={43},
	number={3},
	pages={1653--1684},
	year={2023},
	publisher={Oxford Univ. Press}
}

@article{hu2023constraint,
	title={A constraint dissolving approach for nonsmooth optimization over the {S}tiefel manifold},
	author={Hu, X. and Xiao, N. and Liu, X. and Toh, K.-C.},
	journal={IMA J. Numer. Anal.},
  volume={44},
  number={6},
  pages={3717--3748},
  year={2024},
  publisher={Oxford University Press}
}

@article{zhang2023riemannian,
	title={A {R}iemannian smoothing steepest descent method for non-{L}ipschitz optimization on embedded submanifolds of {$\mathbb{R}^n$}},
	author={Zhang, C. and Chen, X. and Ma, S.},
	journal={Math. Oper. Res.},
	volume={49},
	number={3},
	pages={1710--1733},
	year={2024},
	doi={10.1287/moor.2022.0286},
	publisher={INFORMS}
}

@article{zou2018selective,
	title={A selective overview of sparse principal component analysis},
	author={Zou, H. and Xue, L.},
	journal={Proc. IEEE},
	volume={106},
	number={8},
	pages={1311--1320},
	year={2018},
	publisher={IEEE}
}

@article{zalcberg2021fair,
	title={Fair principal component analysis and filter design},
	author={Zalcberg, G. and Wiesel, A.},
	journal={IEEE Trans. Signal Process.},
	volume={69},
	pages={4835--4842},
	year={2021},
	publisher={IEEE}
}

@article{hardoon2011sparse,
	title={Sparse canonical correlation analysis},
	author={Hardoon, D. R. and Shawe-Taylor, J.},
	journal={Mach. Learn.},
	volume={83},
	pages={331--353},
	year={2011},
	publisher={Springer}
}

@article{liu2024survey,
	author={Liu, Y.-F. and Chang, T.-H. and Hong, M. and Wu, Z. and So, A. M.-C. and Jorswieck, E. A. and Yu, W.},
	journal={IEEE J. Sel. Areas Commun.},
	title={A survey of recent advances in optimization methods for wireless communications},
    pages={2992-3031},
    volume={42},
	number={11},
	year={2024}
}

@inproceedings{spielman2012exact,
  title={Exact recovery of sparsely-used dictionaries},
  author={Spielman, Daniel A and Wang, Huan and Wright, John},
  booktitle={Proceedings of the 25th Annual Conference on Learning Theory},
  volume={23},
  series={Proceedings of Machine Learning Research},
  pages={37.1--37.18},
  year={2012},
  publisher={PMLR},
  url={https://proceedings.mlr.press/v23/spielman12.html}
}

@article{demanet2014scaling,
  title={Scaling law for recovering the sparsest element in a subspace},
  author={Demanet, Laurent and Hand, Paul},
  journal={Information and Inference: A Journal of the IMA},
  volume={3},
  number={4},
  pages={295--309},
  year={2014},
  publisher={OUP}
}

@article{qu2014finding,
  title={Finding a sparse vector in a subspace: Linear sparsity using alternating directions},
  author={Qu, Qing and Sun, Ju and Wright, John},
  journal={Advances in Neural Information Processing Systems},
  volume={27},
  pages={3401-3409},
  year={2014}
}

@article{sun2016complete,
  title={Complete dictionary recovery over the sphere I: Overview and the geometric picture},
  author={Sun, Ju and Qu, Qing and Wright, John},
  journal={IEEE Transactions on Information Theory},
  volume={63},
  number={2},
  pages={853--884},
  year={2017},
  doi={10.1109/TIT.2016.2632162},
  publisher={IEEE}
}

@article{sun2016complete2,
  title={Complete dictionary recovery over the sphere II: Recovery by {R}iemannian trust-region method},
  author={Sun, Ju and Qu, Qing and Wright, John},
  journal={IEEE Transactions on Information Theory},
  volume={63},
  number={2},
  pages={885--914},
  year={2017},
  doi={10.1109/TIT.2016.2632149},
  publisher={IEEE}
}

@article{zheng2024smoothed,
  title={Smoothed robust phase retrieval},
  author={Zheng, Zhong and Xue, Lingzhou},
  journal={arXiv preprint arXiv:2409.01570},
  year={2024}
}

\newpage
\appendix
\section{Alternative Subproblem Solver When \texorpdfstring{$c(\cdot)$}{c(.)} Is the Identity Mapping}\phantomsection\label{subsec:assn}
In this section, we focus on the situation that $m=n, c(z) = z,\forall z\in\mathbb{R}^n$. In this case, the first-order Algorithm \ref{Nesterov_erg} might not be the best choice. Thus, we introduce the adaptive regularized semi-smooth Newton's method (ASSN) for solving \eqref{eq_imanpl}. It is a second-order algorithm proposed by \cite{xiao2018regularized} and used by \cite{chen2020proximal} in their manifold proximal gradient method (ManPG).

Following the notations in Section \ref{subsec:suff_stopping}, we can also rewrite the optimization problem in \eqref{eq_imanpl} as
\begin{equation}
	\min_{\tilde{x}\in\mathbb{R}^n} \tilde{H}_k(\tilde{x}) := f(z_k) + c_k^\top \tilde{x} + h(\tilde{x} + d_k) + \frac{1}{2t_k}\|\tilde{x}\|_2^2,\ s.t.\ C_k\tilde{x} = 0.
\end{equation}
The Lagrangian function is as follows:
\begin{equation}
	\tilde{L}_k(\tilde{x};\mu) := \tilde{H}_k(\tilde{x}) - \mu^\top C_k\tilde{x}. 
\end{equation}
Thus, the dual problem is
\begin{equation}
	\max_{\mu\in\mathbb{R}^{\hat{n}}} \tilde{D}_k(\mu) := \min_{\tilde{x}\in\mathbb{R}^n}\tilde{L}_k(\tilde{x};\mu). 
\end{equation}
For a fixed outer iteration $k$, the primal minimizer associated with a multiplier is
\[
\tilde{x}_k(\mu)=\operatorname{prox}_{t_kh}(d_k-t_kc_k+t_kC_k^\top\mu)-d_k,
\qquad \mathcal{R}_k(\mu):=t_k^{-1}C_k\tilde{x}_k(\mu)=-t_k^{-1}\nabla\tilde{D}_k(\mu).
\]
As in Section 4.2 of \cite{chen2020proximal}, the dual is smooth and concave, although it need not be strongly concave. Firm nonexpansiveness of the proximal mapping and $C_kC_k^\top=I_{\hat n}$ give $\|\mathcal{R}_k(\mu)-\mathcal{R}_k(\nu)\|_2^2\leq\langle\mathcal{R}_k(\mu)-\mathcal{R}_k(\nu),\mu-\nu\rangle$. Hence $\mathcal{R}_k$ is monotone and $I-\mathcal{R}_k$ is firmly nonexpansive, and thus $1/2$-averaged. The equality-constrained strongly convex subproblem has a unique primal minimizer $x_k^\star=S_{t_k}(z_k)-z_k$ and a dual maximizer $\mu_k^\star$, since $h$ is finite and continuous and the linear constraints are feasible. For the convergence argument, apply Algorithm 1 of \cite{xiao2018regularized} to $\mathcal{R}_k$, retaining its Newton residual test, projection and fixed-point safeguards, and inexact linear-solve condition. The normalization leaves the roots unchanged and makes Assumption 3.1 of that reference applicable for every $t_k>0$; its Theorem 3.4 yields $\|\mathcal{R}_k(\mu_k^j)\|_2\to0$. Put $\tilde{x}_k^j=\tilde{x}_k(\mu_k^j)$. Primal optimality gives $c_k+\tilde{x}_k^j/t_k+v_k^j-C_k^\top\mu_k^j=0$ for some $v_k^j\in\partial h(\tilde{x}_k^j+d_k)$ with $\|v_k^j\|_2\leq L_h$, and therefore
\[
\mu_k^j=\mathcal{R}_k(\mu_k^j)+C_kc_k+C_kv_k^j.
\]
Thus the multiplier sequence is bounded. Comparing this optimality equation with the one at $(x_k^\star,\mu_k^\star)$ and using monotonicity of $\partial h$ gives
\[
\frac{\|\tilde{x}_k^j-x_k^\star\|_2^2}{t_k}
\leq\langle\mu_k^j-\mu_k^\star,C_k\tilde{x}_k^j\rangle
\leq t_k\|\mu_k^j-\mu_k^\star\|_2\|\mathcal{R}_k(\mu_k^j)\|_2\longrightarrow0.
\]
Consequently, for $\hat{x}_k^j=\Proj_{\mathrm{T}_{z_k}\mathcal{M}}(\tilde{x}_k^j)$, we obtain
\begin{equation}\label{eq_limit_dual}
\tilde{x}_k^j\longrightarrow x_k^\star,\qquad
\hat{x}_k^j\longrightarrow x_k^\star,\qquad C_k\tilde{x}_k^j\longrightarrow0.
\end{equation}
Moreover, $\tilde{D}_k(\mu_k^j)=\tilde{H}_k(\tilde{x}_k^j)-(\mu_k^j)^\top C_k\tilde{x}_k^j$, so continuity of $\tilde{H}_k$ and boundedness of $\mu_k^j$ imply
\[
F_{t_k}(\hat{x}_k^j+z_k;z_k)-\tilde{D}_k(\mu_k^j)
=\tilde{H}_k(\hat{x}_k^j)-\tilde{H}_k(\tilde{x}_k^j)+(\mu_k^j)^\top C_k\tilde{x}_k^j\longrightarrow0.
\]
This conclusion requires neither a unique dual maximizer nor a global error bound. To exploit semismooth Newton steps, we assume that $\operatorname{prox}_{t_kh}$ is semismooth; this property is not implied by convexity and Lipschitz continuity alone. In particular, SPCA in \eqref{eq_spca} has $z=\operatorname{vec}(U)$, $c(z)=z$, and $h(z)=\kappa\|z\|_1$, whose proximal mapping is componentwise soft thresholding and is piecewise affine, hence strongly semismooth. Thus SPCA is a special case of this construction for every positive $t_k$.

Next, we introduce our new subproblem termination conditions designed for ASSN. Knowing that ASSN does not guarantee $C_k\tilde{x}_k^j = 0$, we denote $\hat{x}_k^j = \Proj_{\mathrm{T}_{z_k}\mathcal{M}}(\tilde{x}_k^j)$ and design sufficient subproblem termination conditions as follows:
\begin{equation}\label{low2}
	F_{t_k}(\hat{x}_k^j+z_k;z_k) - \tilde{D}_k(\mu_k^j)\leq \rho_l\left(F(z_k) - F_{t_k}(\hat{x}_k^j+z_k;z_k)\right), \rho_l>0.
\end{equation}
\begin{equation}\label{high2}
	F_{t_k}(\hat{x}_k^j+z_k;z_k) - \tilde{D}_k(\mu_k^j)\leq \frac{\rho_h}{2t_k}\|\hat{x}_k^j\|_2^2, \rho_h\in (0,1/4).
\end{equation}
When (IT) = (LACC), we terminate ASSN when it finds $\mu_k^j$ such that \eqref{low2} holds. When (IT) = (HACC), we terminate ASSN when it finds $\mu_k^j$ such that \eqref{high2} holds. Then, we set $\tilde{z}_{k+1} = \hat{x}_k^j+z_k$. By weak duality, \eqref{low2} implies \eqref{low-high0}-(LACC) and \eqref{high2} implies \eqref{low-high0}-(HACC). At a nonstationary outer iterate, $x_k^\star\neq0$: the limiting right-hand sides of both stopping tests are strictly positive, while the primal-dual gap tends to zero. Therefore the safeguarded solver reaches either chosen stopping condition after finitely many iterations; an exactly stationary outer iterate requires no further update. 

Next, we discuss \cite{huang2023inexact,huang2022riemannian2} that also use adaptive stopping conditions when solving the subproblem \eqref{eq_imanpl} via Newton's method. Different from our methods that use the primal-dual gap for adaptive stopping conditions, their methods use $\|\nabla \tilde{D}_k(\mu_k^j)\|_2$. In addition, 
\cite{huang2023inexact}'s condition only works for sufficiently small $t_k$, and \cite{huang2022riemannian2}'s condition is
$$\|\nabla \tilde{D}_k(\mu_k^j)\|_2\leq \sqrt{4t_k^2L_h^2 + \|\hat{x}_k^j\|_2^2/2} - 2t_kL_h,$$
which depends on $L_h$ that might be unknown. This demonstrates better practical adaptiveness of our \eqref{low2} and \eqref{high2}.

\section{Numerical Experiments on Sparse Principal Component Analysis (SPCA)}\phantomsection\label{sec:numerical_spca}
In this section, we apply our IManPL algorithm (Algorithm \ref{alg:IManPL-ls}) with the subproblem \eqref{eq_imanpl} solved by the
ASSN implementation described below for the SPCA problem \eqref{eq_spca}. Here $z=\operatorname{vec}(U)$, $c(z)=z$, $f(z)=-\|AU\|_F^2$, and $h(z)=\kappa\|z\|_1$. Assumption \ref{ass:problem} holds with $L_f=2\|A\|_2^2$, $L_h=\kappa\sqrt{Nr}$, and $L_c=0$, and the Stiefel manifold is compact. The soft-thresholding proximal mapping also satisfies the semismoothness condition in Appendix \ref{subsec:assn}. We compare IManPL with ManPG \citep{chen2020proximal}. ManPG\footnote{Readers can find the code for \cite{chen2020proximal} in \url{https://github.com/chenshixiang/ManPG}.} can be treated as a special situation of ManPL \citep{wang2022manifold} with $c(\cdot)$ being the identity mapping and the subproblems \eqref{eq_imanpl} solved by ASSN that tries to reach a high precision. ManPG also implements \eqref{eq_imanpl} and \eqref{eq_retraction}, using Armijo backtracking line search to determine $\alpha_k$. Both ManPG and IManPL use the polar decomposition for the retraction operation in \eqref{eq_retraction}. 
We initialize both algorithms by the singular value decomposition of a random $N\times r$ matrix whose elements independently follow the standard Gaussian distribution. In addition, when IManPL solves \eqref{eq_imanpl} via ASSN, all the implementation details, except the stopping conditions \eqref{low2} or \eqref{high2}, follow the code of \cite{chen2020proximal}. This is the practical Newton-step implementation described in their Section 6.1, which always takes the semismooth Newton step. The convergence argument in Appendix \ref{subsec:assn} concerns the normalized residual and the globally safeguarded version of ASSN; it does not assert a global convergence guarantee for the practical implementation without those safeguards. Next, we introduce the step sizes $\{t_k\}_{k\in\mathbb{N}}$ that are used in \eqref{eq_imanpl} for both methods. We adopt the design in \cite{chen2020proximal}. $t_0=1/(2\|A\|_2^2)$ and \begin{equation}\label{eq_adaptive_stepsize_spca}
	t_{k+1} = \begin{cases}
		t_kv, & \alpha_k = 1,\\
		\max\{t_0,t_k/v\}, & \alpha_k < 1
	\end{cases}
\end{equation} 
for any $k\in\mathbb{N}$, where $v\geq 1$. When $v=1$, the step sizes are fixed. When $v>1$, the step sizes are adaptive. Later on, we use the algorithm name ManPG for the case $v=1$ and the name ManPG-Ada for the case $v = 1.01$. We will also set $v = 1.01$ for IManPL and use the names IManPL-ASSN-(LACC) and IManPL-ASSN-(HACC)
to highlight the subproblem solver ASSN and 
the choices for (IT). In what follows, we conduct numerical experiments to compare the CPU time of ManPG, ManPG-Ada, and the two versions of IManPL-ASSN. Since ASSN is a second-order algorithm, we compare CPU time, the number of main iterations, and the mean subproblem iterations for solving \eqref{eq_imanpl} across candidate methods. The results show that IManPL-ASSN has the best efficiency, and using (LACC) and (HACC) performs similarly. IManPL-ASSN consistently outperforms ManPG. In addition, IManPL-ASSN is more efficient than ManPG-Ada under relatively large $\kappa$, and they perform similarly under small $\kappa$.

Next, we introduce our simulated experiment environment that follows Section 6.3 of \cite{chen2020proximal}. The random data matrices $A\in\mathbb{R}^{N_1\times N}$ are generated in the following way. We
first generate a random matrix $\tilde{A}$ such that each element independently follows the standard Gaussian distribution, then shift the columns of $\tilde{A}$
so that their mean is equal to 0, and lastly, normalize the columns so that their Euclidean norms are equal
to one. In all tests, $N_1 = 500, N=1000$ and we test $r\in\{10,50\},\kappa\in\{0.1,0.3\}$. For a given combination of $r,\kappa$, we generate 10 replications. In each replication, we use the same initialization for all the candidate algorithms. ManPG follows the termination condition in the code of \cite{chen2020proximal}, and other algorithms terminate when they reach a smaller objective function value than ManPG.

Tables \ref{table:single_compare_cputime1}, \ref{table:compare_cputime_new_group2}, \ref{table:compare_cputime_new_group_4}, and  \ref{table:compare_cputime_new_group_5} provide the experiment results that compare the CPU time, the main iteration numbers, and the mean subproblem iteration numbers of the candidate algorithms. First, we discuss the common conclusions. These tables demonstrate that adaptive step sizes can improve performance: all other methods require fewer main iterations than ManPG, and ManPG-Ada outperforms ManPG in CPU time. This conclusion is consistent with the numerical experiments in \cite{chen2020proximal}. In addition, IManPL-ASSN and ManPG-Ada require similar numbers of main iterations, so we do not need to solve the subproblems to a very high accuracy. In addition, (LACC) and (HACC) perform similarly.

Next, we discuss the differences in results between ManPG-Ada and IManPL-ASSN. When $\kappa = 0.1$ (Tables \ref{table:single_compare_cputime1} and \ref{table:compare_cputime_new_group_4}), which means that the sparsity level is relatively low, the mean subproblem iteration numbers are similar between ManPG-Ada and IManPL-ASSN for using either (LACC) or (HACC). This leads to their similar CPU time. When $\kappa = 0.3$ (Tables \ref{table:compare_cputime_new_group2} and \ref{table:compare_cputime_new_group_5}), which means that the sparsity level is relatively high, IManPL-ASSN requires fewer subproblem iterations for using either (LACC) or (HACC) and outperforms ManPG-Ada in terms of CPU time.
\begin{table}[H]
	\centering
\begin{tabular}{|c|c|c|c|}
    \hline
    & main iterations & mean sub-iterations & CPU Time \\
    \hline
    ManPG  & 1330.20 (298.77) & 1.07 (0.02) & 2.47 (0.78) \\
    ManPG-Ada  & 621.00 (418.98) & 1.48 (0.15) & 1.72 (1.36) \\
    IManPL-ASSN-(LACC) & 515.10 (133.76) & 1.00 (0.00) & 1.46 (0.77) \\
    IManPL-ASSN-(HACC) & 515.10 (133.76) & 1.00 (0.00) & 2.33 (3.20) \\
    \hline
\end{tabular}
	\caption{Experiments for SPCA with $r = 10,\kappa = 0.1$. Results for the replications are reported in the form ``mean (standard deviation)". For one replication, ``mean sub-iterations" is short for the mean iterations for solving all the subproblems \eqref{eq_imanpl}.}
	\label{table:single_compare_cputime1}
\end{table}

\begin{table}[H]
	\centering
\begin{tabular}{|c|c|c|c|}
    \hline
    & main iterations & mean sub-iterations & CPU Time \\
    \hline
    ManPG  & 1435.10 (317.72) & 1.29 (0.05) & 3.67 (2.38) \\
    ManPG-Ada  & 294.70 (47.93) & 3.12 (0.23) & 1.10 (0.95) \\
    IManPL-ASSN-(LACC) & 283.80 (47.71) & 1.22 (0.11) & 0.66 (0.19) \\
    IManPL-ASSN-(HACC) & 283.90 (46.83) & 1.23 (0.12) & 0.66 (0.19) \\
    \hline
\end{tabular}
	\caption{Experiments for SPCA with $r = 10,\kappa = 0.3$. Results for the replications are reported in the form ``mean (standard deviation)". For one replication, ``mean sub-iterations" is short for the mean iterations for solving all the subproblems \eqref{eq_imanpl}.}
	\label{table:compare_cputime_new_group2}
\end{table}

\begin{table}[H]
	\centering
 \begin{tabular}{|c|c|c|c|}
    \hline
    & main iterations & mean sub-iterations & CPU Time \\
    \hline
    ManPG  & 3625.70 (1011.42) & 1.22 (0.07) & 73.39 (22.63) \\
    ManPG-Ada  & 1488.20 (858.57) & 1.95 (0.10) & 21.04 (11.25) \\
    IManPL-ASSN-(LACC) & 1459.10 (899.43) & 1.53 (0.04) & 20.86 (14.62) \\
    IManPL-ASSN-(HACC) & 1420.70 (801.29) & 1.59 (0.05) & 19.24 (10.76) \\
    \hline
\end{tabular}
	\caption{Experiments for SPCA with $r = 50,\kappa = 0.1$. Results for the replications are reported in the form ``mean (standard deviation)". For one replication, ``mean sub-iterations" is short for the mean iterations for solving all the subproblems \eqref{eq_imanpl}.}
	\label{table:compare_cputime_new_group_4}
\end{table}

\begin{table}[H]
	\centering
\begin{tabular}{|c|c|c|c|}
    \hline
    & main iterations & mean sub-iterations & CPU Time \\
    \hline
    ManPG  & 1450.90 (318.20) & 3.83 (0.66) & 42.66 (5.29) \\
    ManPG-Ada  & 299.20 (37.69) & 13.87 (2.20) & 31.16 (3.44) \\
    IManPL-ASSN-(LACC) & 295.30 (68.40) & 1.83 (0.12) & 3.55 (0.75) \\
    IManPL-ASSN-(HACC) & 277.90 (60.89) & 1.86 (0.13) & 3.35 (0.72) \\
    \hline
\end{tabular}
	\caption{Experiments for SPCA with $r = 50,\kappa = 0.3$. Results for the replications are reported in the form ``mean (standard deviation)". For one replication, ``mean sub-iterations" is short for the mean iterations for solving all the subproblems \eqref{eq_imanpl}.}
	\label{table:compare_cputime_new_group_5}
\end{table}

\end{document}